\documentclass[bj]{imsart}

\RequirePackage[OT1]{fontenc}
\RequirePackage{amsthm,amsmath,natbib}
\RequirePackage[colorlinks,citecolor=blue,urlcolor=blue]{hyperref}
\usepackage[english]{babel}
\usepackage{amsfonts,amssymb,amsthm,amsmath,MnSymbol}
\arxiv{arXiv:1902.01136}

\startlocaldefs
\numberwithin{equation}{section}
\theoremstyle{plain}
\newtheorem{theorem}{Theorem}[section]
\newtheorem{corollary}{Corollary}[section]

\newtheorem{proposition}{Proposition}[section]
\newtheorem{remark}{Remark}[section]
\newtheorem{example}{Example}[section]
\newtheorem{definition}{Definition}[section]
\endlocaldefs

\newcommand{\cd}{\ensuremath{\rightsquigarrow}}

\newcommand{\X}{\mathfrak{X}}
\newcommand{\x}{\mathbf{x}}
\newcommand{\sgn}{\mbox{\rm sgn}}
\newcommand{\R}{\ensuremath{\mathbb{R}}}

\DeclareMathOperator*{\esssup}{ess\,sup}
\DeclareMathOperator*{\amp}{amp}
\newcommand{\D}{\ensuremath{\mathcal{D}}}
\def \d{\text{\rm d}}
\def \Emp{\mathbb{E}}
\def \E{\text{\rm E}}
\def \P{\text{\rm P}}
\def \Q{\text{\rm Q}}
\DeclareMathOperator*{\card}{Card}
\newcommand\inp[2]{\langle #1, #2 \rangle}
\def \B{\mathbb{B}}

\usepackage{color}
\definecolor{violet}{rgb}{0.7,0,0.6}

\begin{document}

\begin{frontmatter}
\title{Directional differentiability for supremum-type functionals: statistical applications}
\runtitle{Directional differentiability for supremum-type functionals}

\begin{aug}
\author{\fnms{Javier} \snm{C\'{a}rcamo}\thanksref{e1}\ead[label=e1,mark]{javier.carcamo@uam.es}}
\author{\fnms{Antonio} \snm{Cuevas}\thanksref{e2}\ead[label=e2,mark]{antonio.cuevas@uam.es}}
\and
\author{\fnms{Luis-Alberto} \snm{Rodr\'{i}guez}\thanksref{e3}%
\ead[label=e3,mark]{luisalberto.rodriguez@uam.es}%
\ead[label=u1,url]{www.foo.com}}

\address{Departamento de Matem\'{a}ticas, Universidad Aut\'{o}noma de Madrid, 28049 Madrid (SPAIN)
\printead{e1,e2,e3}}


\runauthor{J. C\'arcamo et al.}

\affiliation{Universidad Aut\'{o}noma de Madrid}

\end{aug}

\begin{abstract}
We show that various functionals related to the supremum of a real function defined on an arbitrary set or a measure space are Hadamard directionally differentiable. We specifically consider the supremum norm, the supremum, the infimum, and the amplitude of a function. The (usually non-linear) derivatives of these maps adopt simple expressions under suitable assumptions on the underlying space.
As an application, we improve and extend to the multidimensional case the results in \cite{Raghavachari} regarding the limiting distributions of Kolmogorov-Smirnov type statistics under the alternative hypothesis. Similar results are obtained for analogous statistics associated with copulas. We additionally solve an open problem about the Berk-Jones statistic proposed by \cite{Jager-Wellner-2004}. Finally, the asymptotic distribution of maximum mean discrepancies over Donsker classes of functions is derived.
\end{abstract}


\end{frontmatter}

\section{Introduction} \label{Section.Introduction}

\textbf{\textsl{The general framework.}} The supremum or uniform norm has been systematically used in statistics to quantify the deviation between an observed phenomenon and a theoretical model.
A well-known case is the goodness-of-fit problem, where the Kolmogorov distance (i.e., the uniform distance between distribution functions) is one of the main tools to carry out the testing procedures.
In this context, the prototypical example is the Kolmogorov-Smirnov test in which the supremum norm of the difference between the empirical distribution function of the sample and the reference distribution function is employed.
The sup-norm has also been notably considered in the literature of almost all fields of statistics such as robustness, density estimation, regression and classification, among others. The reason for the extensive use of this distance might rely on different factors: it has a clear and simple interpretation; it takes into account the global behaviour of the functions; and, in general, it is easy to compute.

The aim of this work is to discuss the (directional) differentiability of the supremum norm --as well as various related functionals that commonly appear in statistics-- viewed as a real functional from the space of bounded functions defined on an arbitrary set or a measure space. We consider the supremum norm, the supremum, the infimum, and the amplitude of a real function. As an application, we use an extended version of the functional Delta method to derive the asymptotic distribution of many statistics that can be expressed in terms of these maps. In this way, we provide a simple and unified approach and the appropriate framework to deal with such type of statistics.

\medskip

\noindent \textbf{\textsl{The problem under study.}} Throughout this work, $\X$ is a  nonempty set and $\ell^\infty(\X)$ is the real Banach space of bounded functions $f:\X\longrightarrow\R$, equipped with the supremum norm, $\Vert f \Vert_\infty:=\sup_{x\in \X}|f(x)|$. 
If additionally $(\X,\mathcal{A},\mu)$ is a measure space, where $\mathcal{A}$ is a $\sigma$-algebra and $\mu$ a positive measure, we denote by $\ell^\infty(\X,\mathcal{A},\mu)$ the set of classes of equivalence of measurable and essentially bounded functions $f:\X\longrightarrow\R$ with the norm $\Vert f\Vert_{\ell^\infty(\mu)}:=\esssup_{x\in \X} |f(x)|$, where
\begin{equation*}
\esssup_{x\in \X} f:= \inf\{  C\in\R : \mu(\{ x\in \X  : f(x) >C \})=0  \}.
\end{equation*}

Important examples of this general setting are $\X=\R^d$ or $\bar\R^d$ ($d\ge 1$), with $\bar{\R}\equiv[-\infty,+\infty]$ the extended real line, and $\X=\X$, a class of real functions. To avoid unnecessary repetitions, unless specifically mentioned, from now on we will only consider the supremum.

For $q\in\ell^\infty(\X)$, the quantity of interest that we want to estimate is  $\phi(q)$, where $\phi$ is any of the following functionals:
\begin{equation}\label{funtionals}
\delta(f):= \Vert f\Vert_\infty,\quad \sigma (f):=\sup_{\X} f,\quad \iota(f):=\inf_{\X} f, \quad \text{and}\quad \alpha (f):=\amp_{\X} f,\quad f \in\ell^\infty(\X),
\end{equation}
with $\amp_{\X} f:=\sup_{\X} f-\inf_{\X} f$ (the amplitude of the function $f$).

We will assume that $q$ can be estimated by $\mathbb{Q}_n$, a random element taking values in $\ell^\infty(\X)$ a.s. satisfying
\begin{equation}\label{weak}
r_n(\mathbb{Q}_n-q)\rightsquigarrow \mathbb{Q}\quad   \text{in } \ell^\infty(\X),\quad \text{as } n\to\infty,
\end{equation}
where $r_n$ is a sequence of real numbers such that $r_n\to\infty$, $\mathbb{Q}$ is a tight Borel random variable in $ \ell^\infty(\X)$, and we use the arrow `$\rightsquigarrow$' to denote the weak convergence of probability measures in the sense of Hoffmann-J{\o}rgensen (see \cite{van der Vaart-Wellner}). The scaling $r_n$ usually goes to infinity as the square root of $n$, but its behaviour could be different in some examples.
In \cite{van der Vaart-Wellner} the theory of weak convergence is developed for a net of probability spaces, that is, a family of spaces indexed by a \emph{directed set}. We recall that a directed set $A$ is a non-empty set with a partial order relation `$\preceq$' satisfying that for every $a,b\in A$, there is $c\in A$ such that $a\preceq c$ and $b\preceq c $. The results obtained in this paper could also be stated in terms of nets. Nevertheless, this generalization is not relevant for the applications considered in this work and it will not be considered in what follows.

For $\phi\in\{\delta,\sigma,\iota, \alpha\}$ in (\ref{funtionals}), we are interested in analyzing the asymptotic behaviour of the normalized estimator of $\phi(q)$, that is, the statistic given by
\begin{equation}\label{problem}
D_n(\phi)\equiv D_\phi (q, \mathbb{Q}_n,r_n):=r_n(\phi(\mathbb{Q}_n)-\phi(q)).
\end{equation}


\medskip

\noindent \textbf{\textsl{Background.}} By the continuous mapping theorem, when $q=0$ (the null function), the weak convergence in (\ref{weak}) directly implies that $D_n(\phi) \cd \phi (\mathbb{Q})$. (Note that in this case `$\cd$' is the usual convergence in distribution of random variables.) This situation often corresponds to the case in which $D_n(\phi)$ is a normalized discrepancy --usually measured in terms of the sup-norm-- for testing the null hypothesis $\text{H}_0:$~$q=0$.
In this setting, the limiting behaviour of $D_n(\phi)$ if $q\ne 0$ provides information regarding the asymptotic power of the underlying testing procedure. The classical result on the asymptotic distribution of the Kolmogorov-Smirnov statistic under the null hypothesis (see, e.g., \cite{van der Vaart}) is a well-known example. It is also worth mentioning the usefulness of this approach for testing composite null
hypotheses such as $\text{H}_0 : q\le 0$. In this case, the limiting behavior of $D_n(\phi)$  when $q\ne 0$ provides information about both asymptotic power (when $q\nleq0$) and asymptotic null behavior (when $q\le0$ and $q=0$).
In \cite{Beare-Moon}, \cite{Seo} and \cite{Beare-Shi}, the focus is on asymptotic null behavior.

Finding the asymptotic distribution of $D_n(\phi)$ in (\ref{problem}) when $q$ is not identically zero is a more challenging problem. So far, this problem has been tackled generally for the sup-norm and some particular choices of the function $q$. To the best of our knowledge, the first remarkable result in this direction was obtained by \cite{Raghavachari}. This author found the asymptotic distribution of the normalized version of the plug-in estimator of $\phi(F-G)$ (for $\phi\in\{\delta,\sigma,\alpha\}$) in the one-sample and two-sample cases when $F$ and $G$ are continuous univariate distribution functions. The results in \cite{Raghavachari} have also been summarized in \citet[Chapter 26]{DasGupta}\color{black}. Over the years, the ideas in \cite{Raghavachari} have been used and replicated by several authors to obtain different results in similar settings.
A non-exhaustive list of these references is:  \cite{Alvarez-Esteban-2012}, \cite{Alvarez-Esteban-2016}, \cite{Freitag-Lange-Munk}, \cite, \cite{Schmoyer}, among others. 
In \cite{Genest-Neslehova}, the authors discussed a test of radial symmetry for copulas in which the key element is the estimation of $\Vert C-\bar C\Vert_\infty$, where $C$ is a bivariate copula and $\bar C$ is its survival copula. \cite{Dette-Volgushev-Bretz} used the same technique to find the asymptotic distribution of the estimator of $\Vert m_1(\beta_1)-m_2(\beta_2) \Vert_\infty$, where $m_1(\beta_1)$ and $m_2(\beta_2)$ are regression functions with parameters $\beta_1$ and $\beta_2$, respectively. In \citet[Theorem 6.1]{Dette-Kokot-Aue} a result in the same spirit as \cite{Raghavachari} is obtained for convergence of suprema of non-centered processes indexed by directed sets (see Remark \ref{Remark.Dette}).


\medskip

\noindent \textbf{\textsl{The proposed methodology.}} In all the previous references the approach used to compute the limiting distributions is based on the direct probabilistic analysis of the considered statistics. For instance, the proofs in \cite{Raghavachari} are essentially based on a careful analysis of the behaviour of the empirical process in the set of points around which the supremum in $\Vert F-G\Vert_\infty$ is attained. However, we explore here an alternative, more general, approach. It is based on the idea that the statistics in (\ref{problem}) have indeed the usual form, suitable to apply the functional Delta method.
Therefore, in light of (\ref{problem}), a direct and intuitive approach to find the asymptotic distribution of $D_n(\phi)$ could be to analyze the differentiability of the maps in (\ref{funtionals}) and use the functional Delta method.  In fact, as it will become evident in this work, looking at the behaviour and analytic properties of the underlying functional is much more enlightening than working directly with the probability distribution of the statistic.


Though there are many possible ways of defining the concept of differentiability for maps between metric or normed spaces, Hadamard differentiability is perhaps the most convenient in this context as it is appropriate for applying the functional Delta method (see \citet[Section 20]{van der Vaart}). However, there are many important examples of maps which are \textit{not} Hadamard differentiable. This is the case of  the functionals in (\ref{funtionals}), which are clearly continuous but non-differentiable. Despite not being fully differentiable, we will show that these maps are \textit{Hadamard directionally differentiable}.
This weaker notion of differentiability was introduced by \cite{Shapiro-1990}. \cite{Shapiro-1991} and \cite{Dumbgen} (see also \cite{Romisch}) independently showed that the Delta method still holds for directional differentiable maps. Recently, this idea has been successfully exploited in the econometric literature; see \cite{Beare-Moon}, \cite{Kaido}, \cite{Seo}, and \cite{Beare-Shi}. \cite{Fang-Santos} illustrate in depth the applicability of the directional differentiability to a wide variety of problems in econometrics. See additionally \cite{Beare-Fang} and \cite{Sommerfeld-Munk}.


\medskip

\noindent\textbf{\textsl{Structure and main results.}} 
In Section \ref{Section.Main} we give the necessary definitions, prove that the maps in (\ref{funtionals}) are Hadamard directional differentiable and determine their 
derivatives under very general assumptions. In particular, this implies that an extended version of the functional Delta method can be applied for these mappings.
As far as the authors know, in the statistical community the Hadamard directional differentiability of the infimum under no additional conditions on the underlying space was first obtained by \citet[Proposition 1]{Romisch}, after a personal communication of P. Lachout in 2004. \citet[Lemma S.4.9]{Fang-Santos} also obtained an expression for the Hadamard directional derivative of the supremum for continuous functions defined on a compact metric space.

 In Section \ref{Section.Main}, besides reviewing the different notions of differentiability and the Delta method, we also obtain several original results.
\begin{enumerate}
\item[(a)] Theorem \ref{Theorem.differentiability} in Section \ref{Subsection.General.Result} 
follows the spirit of \citet[Proposition 1]{Romisch} though our proof is slightly different and we include the supremum norm (not covered in \cite{Romisch}). In the rest of Section \ref{Section.Main} we rely on this result to obtain simplified expressions for the derivatives of the mappings in (\ref{funtionals}) when the space $\X$ is endowed with additional structure.

\item[(b)] In Section \ref{Subsection.Compact.Spaces} we assume that $\X$ is a compact metric space. The main novelty here is that the involved functions are \textit{not} required to be continuous (and we continue to deal with the supremum norm). \citet[Lemma S.4.9]{Fang-Santos} is obtained as a particular case. 

\item[(c)] In Sections \ref{Subsection.Totally.bounded} and \ref{Subsection.Weakly.compact} we consider the case in which $\X$ is a totally bounded me\-tric space and a weakly compact subset of a Banach space, respectively. To the best of our knowledge, the corresponding differentiability results are new in the literature.

\item[(d)] In Section \ref{Subection.Skorohod.Space} we analyze in detail the situation in which $\X=\bar \R^d$ and the functions belong to $\D(\bar\R^d)\equiv$ the extension of the Skorohod space in $[0,1]^d$ (introduced in \cite{Neuhaus}) to the whole $\bar\R^d$. The space $\D(\bar\R^d)$ is an important subspace of $\ell^\infty (\bar\R^d)$ as it includes the paths of many well-known stochastic processes with jumps in their paths such as multivariate empirical processes. Hence, the functions in $\D(\bar\R^d)$ are not necessarily continuous and the expressions of the derivatives of the maps are new.
\end{enumerate}

The versatility of the proposed methodology is illustrated in depth in Sections \ref{Section.Distribution.Functions}-\ref{Setion.Maximum.mean.discrepancies}, where we derive the asymptotic distribution of various statistics with no additional effort. We base the results on the directional differentiability of the functionals and the weak convergence of the underlying stochastic processes. Hence, this unifying approach allows us to reduce a usually difficult statistical problem to a much simpler analytical question related to the directional differentiability of the corresponding functional. Using these ideas, we obtain the following applications: In Section \ref{Section.Distribution.Functions} we extend and give simpler and shorter proofs of the results in \cite{Raghavachari} both in the one-sample and two-sample cases. The extension is carried out in different directions. Firstly, no assumption on the involved distribution functions is necessary to derive the asymptotic results. In contrast, in \cite{Raghavachari} the continuity of the distribution functions is required. Secondly, the results are obtained in a multidimensional setting. We note that the proofs are very simple (compared with those in \cite{Raghavachari}) because they just rely on the analysis of the differentiability of the functionals and the convergence of the associated processes separately. It should be further remarked that those works that have used the results and ideas in \cite{Raghavachari} were forced to impose the continuity of the involved functions as an assumption in their statements; see for instance \citet[Equation (11)]{Alvarez-Esteban-2016}, \citet[Section 2]{Freitag-Lange-Munk} or \citet[Assumption 7.4.]{Dette-Volgushev-Bretz}.
The regularity limitation of working with continuous functions is not mathematically aesthetic and it is in fact unnecessary, as we will show in this paper. Moreover, in Section \ref{Section-Copula} we will extend these results to copulas. Also, in Section \ref{Setion.Jager.Weller}, we apply this technique to solve an open question by \cite{Jager-Wellner-2004} related to the Berk-Jones statistic. Finally, in Section \ref{Setion.Maximum.mean.discrepancies} we derive the asymptotic distribution for the plug-in estimators of maximum mean discrepancies with respect to a Donsker class.

The main results of this paper can also be applied to find the asymptotic distribution of the empirical risk over Donsker classes of functions and estimators of kernel distances. These applications are not included in the present paper due to the limited space available and they will be developed in future works.


\section{Main results}\label{Section.Main}

In this section we introduce the definitions of directional differentiability of maps between Banach spaces, recall an extended version of the Delta method for these mappings, and discuss the analytic properties of the functionals introduced in Section \ref{Section.Introduction} according to the mathematical structure of $\X$.

\subsection{Directional differentiability and the Delta method}\label{Subsection.Differentiability}

In many situations it is common to face the problem of estimating a transformation, $\phi(\theta)$, of a (possibly infinite-dimensional) parameter $\theta$. Typically, $\theta$ is unknown but can be estimated by means of $T_n$ and $\phi$ is a map defined in a metric space. If $\phi$ is smooth enough in a local neighborhood of $\theta$ --for instance, differentiable at $\theta$ in a precise sense-- the asymptotic distribution of (the normalized version of) $\phi(T_n)$ can be determined by expanding $\phi$ around $\theta$ and using an invariance principle for $T_n$ in the underlying metric space. Of course, this is the key idea behind the \textit{(functional) Delta method}, one of the most frequently used methodologies in statistics to compute the limiting distribution of an estimator of a quantity of interest (see \citet[Section 3.9]{van der Vaart-Wellner}). This technique is specially fruitful when dealing with the popular plug-in estimators, which, by construction, are functions of the empirical distribution function of the observed sample. In such cases, the powerful theory of weak convergence of empirical processes provides the suitable mathematical machinery to determine the asymptotic behaviour of this kind of estimators (see \cite{Gine-Nickl}).

We start with the notion of G\^{a}teaux directional differentiability.

\begin{definition}\label{Definition.Gateaux}
Let $\mathcal{D}$ and $\mathcal{E}$ be real Banach spaces with norms $\Vert \cdot\Vert_\mathcal{D}$ and $\Vert \cdot\Vert_\mathcal{E}$, respectively. A map $\phi:\mathcal{D}\longrightarrow \mathcal{E}$ is said to be {\rm G\^{a}teaux directionally differentiable} at $\theta\in \mathcal{D}$ tangentially to a set $\mathcal{D}_0\subset \mathcal{D}$ if there exists a map $\phi^\prime_\theta: \mathcal{D}_0 \longrightarrow \mathcal{E}$ such that
\begin{equation}\label{Gateaux}
\left\Vert  \frac{\phi(\theta+t_n h)-\phi(\theta)}{t_n}-\phi^\prime_\theta(h)\right\Vert_\mathcal{E}\to 0,
\end{equation}
for all $h\in \mathcal{D}_0$ and all sequences $\{t_n\}\subset \R$ such that $t_n\downarrow 0$.
\end{definition}

It is well-known that G\^{a}teaux differentiability is too weak for the Delta method to hold. To solve this problem, the directions along which we approach to $\phi(\theta)$ in (\ref{Gateaux}) have to be allowed to change with $n$. This naturally leads to the concept of Hadamard directional differentiability. We follow \cite{Shapiro-1990} for the next definition.

\begin{definition}\label{Definition.Hadamard}
In the context of the previous definition, we say that $\phi:\mathcal{D}\longrightarrow \mathcal{E}$ is  {\rm Hadamard directionally differentiable} at $\theta\in \mathcal{D}$ tangentially to a set $\mathcal{D}_0\subset \mathcal{D}$ if there exists a map $\phi^\prime_\theta: \mathcal{D}_0 \longrightarrow \mathcal{E}$ such that
\begin{equation}\label{Hadamard}
\left\Vert  \frac{\phi(\theta+t_n h_n)-\phi(\theta)}{t_n}-\phi^\prime_\theta(h)\right\Vert_\mathcal{E}\to 0,
\end{equation}
for all $h\in \mathcal{D}_0$ and all sequences $\{t_n\}\subset \R$ and $\{h_n\}\subset \mathcal{D}$ such that $t_n\downarrow 0$ and $\Vert h_n-h\Vert_\mathcal{D}\to 0$.
\end{definition}

Obviously, the Hadamard directional differentiability condition \eqref{Hadamard} is stronger than the G\^{a}teaux notion \eqref{Gateaux}. The only difference between the directional and the usual differentiability is that the derivative $\phi^\prime_\theta$ is no longer required to be linear in Definitions \ref{Definition.Gateaux} and \ref{Definition.Hadamard}. Nevertheless, if equation (\ref{Hadamard}) is satisfied, then $\phi^\prime_\theta$ is continuous and positive homogeneous of degree 1 (see \citet[Proposition 3.1]{Shapiro-1990}).

\begin{remark}\label{Remark.Gateaux.Hadamard}
If $\phi$ is as in Definitions \ref{Definition.Gateaux}, and additionally $\phi$ is locally Lipschitz, i.e., there exists a constant $C>0$ such that $\Vert \phi(f)-\phi(g)\Vert_{\mathcal{E}}\le C \Vert f-g\Vert_{\mathcal{D}}$, for all $f,g\in \mathcal{D}$ in a neighborhood of each point of $\mathcal{D}$, then Hadamard directional differentiability is equivalent to the G\^{a}teaux one (see \citet[Proposition 3.5]{Shapiro-1990}). This condition is satisfied by $\delta, \sigma, \iota, \alpha: \ell^\infty(\X) \longrightarrow \R$ defined in (\ref{funtionals}). Hence, to check that the maps considered in Section \ref{Section.Introduction} are Hadamard directionally differentiable at $f\in  \ell^\infty(\X)$ we only need to show G\^{a}teaux directional differentiability.
\end{remark}

The important fact about Hadamard directional differentiability is that it is the crucial condition to ensure the validity of the following \textit{extended (functional) Delta method}.

\begin{proposition}\label{Proposition.delta}
Let $\mathcal{D}$ and $\mathcal{E}$ be Banach spaces and $\phi:\mathcal{D}_\phi\subset \mathcal{D} \longrightarrow \mathcal{E}$, where $\mathcal{D}_\phi$ is the domain of $\phi$. Assume that $\phi$ is Hadamard directionally differentiable at $\theta\in \mathcal{D}_\phi$ tangentially to a set $\mathcal{D}_0\subset \mathcal{D}$. For some sample spaces $\Omega_n$, let $T_n:\Omega_n\longrightarrow \mathcal{D}_\phi$ be maps such that $r_n(T_n-\theta)\rightsquigarrow T$, for some sequence of numbers $r_n\to\infty$ and a random element $T$ that takes values in $\mathcal{D}_0$. Then, $r_n( \phi(T_n)-\phi(\theta))\rightsquigarrow \phi^\prime_\theta(T)$. If additionally $\phi^\prime_\theta$ can be continuously extended to $\mathcal{D}$, then we have that $r_n( \phi(T_n)-\phi(\theta))=\phi^\prime_\theta(r_n(T_n-\theta))+o_\P(1)$.
\end{proposition}

\begin{remark}\label{Remark.Proof.Delta}
The detailed proof of Proposition \ref{Proposition.delta} can be found in \citet[Theorem 2.1]{Shapiro-1991} (see also \citet[Theorem 1]{Romisch} or \citet[Theorem 2.1]{Fang-Santos}), but it is essentially the same one as for the traditional Delta method; see \citet[Theorem 20.8]{van der Vaart}. The key idea is to apply the  \textit{extended continuous mapping theorem} (\citet[Theorem 1.11.1]{van der Vaart-Wellner}) to the sequence of functionals defined by $\phi_n(h):=r_n(\phi(\theta+r_n^{-1}h)-\phi(\theta))$, $n\in \mathbb{N}$.
\end{remark}


In the present context, let us assume that $\theta_n\to\theta$ and $r_n(T_n-\theta_n)\rightsquigarrow T$, and we want to determine conditions so that $r_n( \phi(T_n)-\phi(\theta_n))\rightsquigarrow \phi^\prime_\theta(T)$. As it is pointed out in \citet[p. 375]{van der Vaart-Wellner}, a stronger form of differentiability is needed to obtain such a ``uniform" version of the Delta method.
\begin{definition}\label{Definition.Hadamard.uniform}
In the context of Definition \ref{Definition.Gateaux}, we say that $\phi:\mathcal{D}\longrightarrow \mathcal{E}$ is {\rm uniformly Hadamard differentiable} at $\theta\in \mathcal{D}$ tangentially to a set $\mathcal{D}_0\subset \mathcal{D}$ if there exists a map $\phi^\prime_\theta: \mathcal{D}_0 \longrightarrow \mathcal{E}$ such that
\begin{equation*}
\left\Vert  \frac{\phi(\theta_n+t_n h_n)-\phi(\theta_n)}{t_n}-\phi^\prime_\theta(h)\right\Vert_\mathcal{E}\to 0,
\end{equation*}
for all $h\in \mathcal{D}_0$ and all sequences $\{t_n\}\subset \R$, $\{\theta_n\}$, $\{h_n\}\subset \mathcal{D}$ such that $t_n\downarrow 0$, $\Vert \theta_n-\theta\Vert_\mathcal{D}\to 0$, and $\Vert h_n-h\Vert_\mathcal{D}\to 0$.
\end{definition}
If $\phi$ is uniformly Hadamard differentiable at $\theta$, $\theta_n\to\theta$ and $r_n(T_n-\theta_n)\rightsquigarrow T$, we still have that $r_n( \phi(T_n)-\phi(\theta_n))\rightsquigarrow \phi^\prime_\theta(T)$; see \citet[Theorem 3.9.5]{van der Vaart-Wellner}.

\subsection{A general result on Hadamard directional differentiability}\label{Subsection.General.Result}

In the next theorem we show that the maps introduced in Section \ref{Section.Introduction} are directionally differentiable at every function of $\ell^\infty(\X)$, where $\X$ is an arbitrary space. In the sequel $\sgn(\cdot)$ denotes the sign function.

\begin{theorem}\label{Theorem.differentiability}
The maps $\delta$, $\sigma$, $\iota$ and $\alpha$ in (\ref{funtionals}) are Hadamard directionally differentiable at every $f\in \ell^\infty(\X)\setminus \{ 0 \}$. For $g\in\ell^\infty(\X)$, their  derivatives are respectively given by

\begin{equation}\label{derivatives}
\begin{aligned}
\delta^\prime_f(g) & = \lim_{\epsilon\downarrow 0} \sup_{A_\epsilon(|f|)} \left( g \cdot \sgn(f)\right),\quad   &  \sigma^\prime_f(g) &= \lim_{\epsilon\downarrow 0} \sup_{A_\epsilon(f)} g,\\[2 mm]
\iota^\prime_f(g) &= \lim_{\epsilon\downarrow 0}  \inf_{B_\epsilon(f)} g ,\quad   &   \alpha^\prime_f(g) &= \lim_{\epsilon\downarrow 0} \Big(\sup_{A_\epsilon(f)} g -\inf_{B_\epsilon(f)} g \Big),
\end{aligned}
\end{equation}
where, for $\epsilon>0$ and $h\in \ell^\infty(\X)$, $A_\epsilon(h)$ and $B_\epsilon(h)$ are the superlevel and sublevel sets of $h$ defined by
\begin{equation*}
A_\epsilon(h):=\Big\{ x\in \X : h(x)\ge \sup_\X h -\epsilon  \Big\} \quad\text{and}\quad B_\epsilon(h):=\Big\{ x\in \X : h(x)\le \inf_\X h +\epsilon  \Big\}.
\end{equation*}
Moreover, if $(\X,\mathcal{A},\mu)$ is a measure space, the result still holds if we substitute the suprema (respectively infima) by essential suprema (respectively infima) with respect to $\mu$.
\end{theorem}

\begin{proof}
We first start with $\sigma$ as the conclusion for the rest of the maps can be derived from this case. Let us fix $f\in \ell^\infty(\X)\setminus  \{0\}$. For $n\in\mathbb{N}$ and each sequence of real numbers $\{s_n\}$ such that $s_n\uparrow \infty$, we consider $\sigma_n(f):\ell^\infty(\X) \longrightarrow\R$ defined by
\begin{equation}\label{sigma-n}
\sigma_n (f,g):=\sup_\X  (s_n f+ g) -s_n\sup_\X f,\quad g\in \ell^\infty(\X).
\end{equation}
From Remarks \ref{Remark.Gateaux.Hadamard} and \ref{Remark.Proof.Delta}, it suffices to show that $\sigma_n (f,g)\to \sigma^\prime_f(g)$, as $n\to\infty$, with $\sigma^\prime_f(g)$ defined in (\ref{derivatives}).
For $\epsilon>0$ and ${x} \notin A_\epsilon(f)$, we have that
\begin{equation*}
s_n f({x})+g({x})-s_n\sup_\X f \le  \sup_\X g - s_n \epsilon.
\end{equation*}
Hence, for all $\epsilon>0$, we obtain that
\begin{equation}\label{bound.1.th1}
\begin{split}
\limsup_{n\to\infty} \sigma_n (f,g) &=  \limsup_{n\to\infty} \Big( \sup_{A_\epsilon(f)}   (s_n f+g)-s_n \sup_\X f     \Big)\\
&\le \sup_{A_\epsilon(f)}  g.
\end{split}
\end{equation}

Conversely, let us define
\begin{equation}\label{h}
h(\epsilon):=\sup_{A_\epsilon(f)} g,\quad \epsilon>0.
\end{equation}
Observe that $h$ is non-decreasing and thus the limit as $\epsilon$ decreases to $0$ exists and, by definition, coincides with $\sigma_f^\prime(g)$.
For each $m\in\mathbb{N}$, there exists $x_m\in A_{1/m}(f)$ satisfying
\begin{equation}\label{sup-sequences}
g({x}_m)\ge h(1/m)-1/m\quad\text{and}\quad f({x}_m)\ge \sup_\X f -1/m.
\end{equation}
From (\ref{sup-sequences}), for each $s_n$, we have that
\begin{equation}\label{bound.2.th1}
\begin{split}
h(1/m)&\le g({x}_m)+1/m\\
&= s_n f({x}_m)+g({x}_m)-s_n f({x}_m)+1/m\\
&\le \sigma_n(f,g)+ (s_n+1)/m.
\end{split}
\end{equation}
Now (\ref{bound.2.th1}) implies that, for all $n\in\mathbb{N}$,
\begin{equation}\label{bound.3.th1}
\lim_{\epsilon\downarrow 0} \sup_{A_\epsilon(f)} g = \lim_{m\to\infty} h(1/m) \le \sigma_n(f,g).
\end{equation}
The proof corresponding to $\sigma$ follows from (\ref{bound.1.th1}) and (\ref{bound.3.th1}).

Now, we consider the map $\delta$ in (\ref{funtionals}). Assume that $f\in \ell^\infty(\X) $ with $\Vert f\Vert_\infty>0$.
For  $g\in \ell^\infty(\X) $, we have to show that  $\delta_n (f,g)\to \delta^\prime_f(g)$, as $n\to\infty$, where $\delta_n (f,g):=\Vert s_n f+ g \Vert_\infty-s_n\,\Vert f \Vert_\infty$ and $s_n\uparrow \infty$. First,
for $\epsilon<\Vert f\Vert_\infty/2$ and $s_n>2\Vert g \Vert_\infty /\Vert f \Vert_\infty$, it is readily checked that $s_n\, |f| +  \sgn(f)\cdot  g \ge 0$ globally on $A_\epsilon(|f|)$. We hence conclude that
\begin{equation*}
\lim_{n\to\infty}  \delta_n (f,g) = \lim_{n\to\infty}  \sigma_n (|f|,g \cdot\sgn(f))=  \sigma_{|f|}^\prime (g \cdot \sgn(f))=\delta_f^\prime(g).
\end{equation*}

The proof for $\iota$ and $\alpha$ follows from the duality between supremum and infimum. Finally, the case in which $\X$ is a measure space  can be treated in a similar way so it is therefore omitted.
\end{proof}

As pointed out in the introduction, \citet[Proposition 1]{Romisch} provides the same result as Theorem \ref{Theorem.differentiability} for the infimum. Obviously, the derivatives of the supremum and amplitude of a function can be derived from the infimum by duality. The additional contribution of Theorem \ref{Theorem.differentiability} is the differentiability of the supremum norm operator, $\delta$. Also, the proof we have included here is slightly different to the one in \citet{Romisch}. The expressions in \eqref{derivatives} will be used throughout Sections \ref{Subsection.Compact.Spaces}-\ref{Subection.Skorohod.Space} to obtain simplified expressions of the derivatives.

Theorem \ref{Theorem.differentiability} ensures that the functionals in \eqref{funtionals} are Hadamard direc\-tio\-nally differentiable. Nevertheless, in general these maps are \textit{not} in uniformly Hadamard differentiable (see Definition \ref{Definition.Hadamard.uniform}) as the following example shows.

\begin{example}\label{Example-not-uniform}
{\rm Let $\X$ be the interval $[0,1]$ in $\R$ and we consider the function $f\equiv 1$. For $x\in [0,1]$ and $n\in\mathbb{N}$, let $f_n(x):= 1+x/n$, $g(x):=1-x$, and $s_n=n$. We have that $f_n\to f$ in $\ell^\infty(\X)$ and it is easy to check that $\sigma_n(f_n,g)=0$, where $\sigma_n$ is given in \eqref{sigma-n}. However, $\sigma_f^\prime(g)=\sup_{[0,1]} g=1$. We conclude that $\sigma$ is not uniformly Hadamard differentiable, and therefore neither are the rest of the maps  in \eqref{funtionals}.
}
\end{example}

Following the same ideas as in the proof of Theorem \ref{Theorem.differentiability}, the following partial result can be proved.
\begin{corollary}
Let $\delta$, $\sigma$, $\iota$ and $\alpha$ be as in (\ref{funtionals}). For each $f$, $g\in \ell^\infty(\X)$ and all sequences $\{t_n\}\subset\R$, $\{f_n\}$, $\{g_n\}\subset\ell^\infty(\X)$ such that $t_n\downarrow0$ , $f_n\to f$ and $g_n\to g$ in $\ell^\infty(\X)$, we have that
\begin{equation}\label{bound.derivatives}
\begin{aligned}
 \limsup_{n\to\infty}\frac{\delta(f_n+t_n g_n)-\delta(f_n)}{t_n}&\le \delta^\prime_f (g),\quad   &  \limsup_{n\to\infty}\frac{\sigma(f_n+t_n g_n)-\sigma(f_n)}{t_n}&\le \sigma^\prime_f (g),\\[2 mm]
\liminf_{n\to\infty}\frac{\iota(f_n+t_n g_n)-\iota(f_n)}{t_n}&\ge \iota^\prime_f (g),\quad   &   \limsup_{n\to\infty}\frac{\alpha(f_n+t_n g_n)-\alpha(f_n)}{t_n}&\le \alpha^\prime_f (g),
\end{aligned}
\end{equation}
where $\delta^\prime_f$, $\sigma^\prime_f$, $\iota^\prime_f$ and $\alpha^\prime_f$ are given in \eqref{derivatives}.
\end{corollary}

In general, the reverse inequalities in \eqref{bound.derivatives} fail to hold because it is not possible to control the term $(\phi(f_n)-\phi(f))/t_n$ ($\phi\in\{\delta,\sigma,\iota, \alpha\}$), for all sequences $\{t_n\}\subset\R$ and $\{f_n\}\subset\ell^\infty(\X)$ such that $t_n\downarrow0$ and $f_n\to f$.

\subsection{Compact metric spaces}\label{Subsection.Compact.Spaces}

In some occasions the limit in $\epsilon$ of the derivatives in (\ref{derivatives}) can be removed. For example, if $\X$ is a compact metric space, the derivatives can be characterized by means of convergent sequences in $\X$ as the following corollary shows. 

\begin{corollary}\label{Corollary.Compact}
In the context of Theorem \ref{Theorem.differentiability}, let us further assume that $(\X,d)$ is a compact metric space. The derivatives in (\ref{derivatives}) can be expressed as

\begin{equation}\label{derivatives-2}
\begin{aligned}
\delta^\prime_f(g) &= \sup_{A_0(|f|)} (g \cdot \sgn(f))^\filledtriangleup_{|f|},\quad   &  \sigma^\prime_f(g)  &= \sup_{A_0(f)} g^\filledtriangleup_f,\\[2 mm]
\iota^\prime_f(g) &= \inf_{B_0(f)} g^\filledtriangledown_f ,\quad   &   \alpha^\prime_f(g) &= \sup_{A_0(f)} g^\filledtriangleup_f -\inf_{B_0(f)} g^\filledtriangledown_f,
\end{aligned}
\end{equation}
where for $h, l \in \ell^\infty(\X)$,
\begin{align}\label{A0B0}
\begin{split}
A_0(h):=& \Big\{ x\in\X: \text{ there exists } \{x_n\}\subset \X \text{ with } x_n\to x\text{ and }h(x_n)\to\sup_\X h  \Big\},\\
B_0(h):=& \Big\{ x\in\X: \text{ there exists } \{x_n\}\subset \X \text{ with } x_n\to x\text{ and }h(x_n)\to\inf_\X h  \Big\},
\end{split}\\
\begin{split}\label{h+h-}
h^\filledtriangleup_l(x):=& \sup \Big\{\limsup_{n\to\infty} h(x_n) : x_n\to x\text{ and } l(x_n)\to\sup_\X l   \Big\},\quad x\in A_0(l),\\
h^\filledtriangledown_l(x):=& \inf \Big\{\liminf_{n\to\infty} h(x_n) : x_n\to x\text{ and } l(x_n)\to\inf_\X l   \Big\},\quad x\in B_0(l).
\end{split}
\end{align}
\end{corollary}

\begin{proof}
We only give a detailed proof for $\sigma$ because the rest of the cases are analogous. We consider the sequence $\{x_m\}$  satisfying (\ref{sup-sequences}) obtained in Theorem \ref{Theorem.differentiability}. As $(\X,d)$ is compact, we can extract a convergent subsequence $x_{m_k}\to x$ in $\X$, as $k\to\infty$. From (\ref{sup-sequences}), we have that $x\in A_0(f)$ and, recalling (\ref{h}), from Theorem \ref{Theorem.differentiability}, we obtain that
\begin{equation}\label{compact-inequality-1}
\sigma_f^\prime(g)=\lim_{k\to\infty} h(1/m_{k})\le \limsup_{k\to\infty} g(x_{m_k})\le g_f^\filledtriangleup(x)\le  \sup_{A_0(f)} g^\filledtriangleup_f.
\end{equation}

In the other direction, let $x\in A_0(f)$ and $\{x_n\}\subset \X$ such that $x_n\to x$ and $f(x_n)\to \sup_\X f$. For each $\epsilon>0$, we have that $x_n\in A_\epsilon(f)$,  for $n$ large enough. We therefore conclude that
\begin{equation}\label{compact-inequality-2}
\limsup_{n\to\infty} g(x_n) \le  \sup_{A_\epsilon(f)} g,\quad \text{ for all }\epsilon>0.
\end{equation}
The conclusion follows from (\ref{compact-inequality-1}), (\ref{compact-inequality-2}) and Theorem \ref{Theorem.differentiability}.
\end{proof}

\begin{remark}
From the proof of Corollary \ref{Corollary.Compact} we see that the result is still valid for sequentially compact topological spaces. 
As this extension is not important for the applications in this work we will omit this framework in the following.
\end{remark}

In the following, if $(\X,d)$ is a metric space we denote by $\mathcal{C}(\X,d)$ the subset of continuous functions in $\ell^\infty(\X)$. We observe that if $g\in \mathcal{C}(\X,d)$, then $g^\filledtriangleup_f(x)=g(x)$ ($x\in A_0(f)$) and $g_f^\filledtriangledown(x)=g(x)$ ($x\in B_0(f)$), where $g_f^\filledtriangleup$ and $g_f^\filledtriangledown$ are defined as in (\ref{h+h-}). If we further assume that $f\in \mathcal{C}(\X,d)$, we have that $A_0(|f|)=M^+(|f|)$, $A_0(f)=M^+(f)$ and $B_0(f)=M^-(f)$, where for $h\in \ell^\infty(\X)$,
\begin{equation}\label{M+M-}
M^+(h):= \Big\{ x\in\X:  h(x)=\sup_\X h  \Big\}\quad\text{and}\quad M^-(h):= \Big\{ x\in\X: h(x)=\inf_\X h  \Big\}.
\end{equation}
This observation yields the following corollary.

\begin{corollary}\label{Corollary.Compact.2}
Let $(\X,d)$ be a compact metric space and let $\delta$, $\sigma$, $\iota$ and $\alpha$ be the maps defined in (\ref{funtionals}). The maps $\sigma$, $\iota$ and $\alpha$ are Hadamard directionally differentiable at any $f\in \ell^\infty(\X)$ tangentially to the set $\mathcal{C}(\X,d)$ with derivatives, for $g\in\mathcal{C}(\X,d)$,
\begin{equation}\label{derivative-3}
\sigma^\prime_f(g)= \sup_{A_0(f)} g, \quad  \iota^\prime_f(g)= \inf_{B_0(f)} g\quad\text{and}\quad \alpha^\prime_f(g)=\sup_{A_0(f)} g -\inf_{B_0(f)} g.
\end{equation}

If additionally $f\in \mathcal{C}(\X,d)\setminus \{ 0 \}$, we have that
\begin{equation}\label{derivative-continuous}
\begin{aligned}
\delta^\prime_f(g) &= \sup_{M^+(|f|)} (g \cdot \sgn(f)),\quad   &  \sigma^\prime_f(g)  &= \sup_{M^+(f)} g,\\[2 mm]
\iota^\prime_f(g) &= \inf_{M^-(f)} g ,\quad   &   \alpha^\prime_f(g) &= \sup_{M^+(f)} g -\inf_{M^-(f)} g,
\end{aligned}
\end{equation}
where $M^+(\cdot)$ and $M^-(\cdot)$ are defined in (\ref{M+M-}).
\end{corollary}

The expression of the derivative $\sigma^\prime_f$ in \eqref{derivative-continuous} for continuous functions defined on a compact metric space has been previously obtained in \citet[Lemma S.4.9]{Fang-Santos}. Observe that equalities in \eqref{derivative-3} are valid even when the function $f$ is not conti\-nuous (as in the more general Corollary \ref{Corollary.Compact}). Note also that $M^+(|f|)$ (respectively, $M^+(f)$ and $M^-(f)$) in (\ref{M+M-}) is the set of extremal points corresponding to the sup-norm (respectively, the supremum and infimum) of $f$.

Another interesting question is to find conditions under which the derivatives of the maps are linear, i.e., the cases in which the mappings are fully Hadamard differentiable.
This kind of results can be traced back to \cite{Banach} (see also \cite{Leonard-Taylor-1983}, \cite{Leonard-Taylor-1985}, and the references therein). In these works the supremum norm differentiability was investigated from the point of view of functional analysis within the space $\mathcal{C}(\X,d)$, with $(\X,d)$ a compact metric space. The following result, a direct consequence of Corollary~\ref{Corollary.Compact.2}, provides general outcomes in a different context. We denote by $\card(A)$ the cardinality of the set $A$.

\begin{corollary}\label{Corollary-Full-differentiability}
Assume that $(\X,d)$ is a compact metric space and let $f\in \ell^\infty(\X)\setminus \{ 0 \}$. Let $A_0(\cdot)$ and $B_0(\cdot)$ be the sets in (\ref{A0B0}). For the maps defined in (\ref{funtionals}) we have that:
\begin{enumerate}
\item[(a)] The map $\delta$ is (fully) Hadamard differentiable at $f$ tangentially to the set $\mathcal{C}(\X,d)$ if and only if $\card(A_0(|f|))=1$ and
$\{ \limsup_{n\to\infty}  \sgn (f(x_n)) : x_n\to x \text{ and } |f(x_n)|\to\Vert f\Vert_\infty   \}=\{c \}$.
In such a case, $\delta^\prime_f(g)=c\, g(x^*)$, where $A_0(|f|)=\{x^*\}$.

\item[(b)] The map $\sigma$ is (fully) Hadamard differentiable at $f$ tangentially to the set $\mathcal{C}(\X,d)$ if and only if $\card(A_0(f))=1$.
In such a case, $\sigma^\prime_f(g)=g(x^+)$, where $A_0(f)=\{x^+\}$.

\item[(c)] The map $\iota$ is (fully) Hadamard differentiable at $f$ tangentially to the set $\mathcal{C}(\X,d)$ if and only if $\card(B_0(f))=1$.
In such a case, $\iota^\prime_f(g)=g(x^-)$, where $B_0(f)=\{x^-\}$.

\item[(d)] The map $\alpha$ is (fully) Hadamard differentiable at $f$ tangentially to the set $\mathcal{C}(\X,d)$ if and only if $\card(A_0(f))=\card(B_0(f))=1$. In such a case, $\alpha^\prime_f(g)=g(x^+)-g(x^-)$, where $A_0(f)=\{x^+\}$ and $B_0(f)=\{x^-\}$.

\end{enumerate}
\end{corollary}

Note that when $f\in \mathcal{C}(\X,d)$, we have that $A_0(|f|)=M^+(|f|)$ in (\ref{M+M-}) and the condition $\card(A_0(|f|))=1$  means that $f$ is a \textit{peaking function}, that is, there exists $x^*\in\X$ such that $|f(x^*)|=\Vert f\Vert_\infty$ and $|f(x^*)|>|f(x)|$, for all $x\in\X$ with $x\ne x^*$.

From a statistical point of view, identifying the cases in which the maps are Hadamard differentiable has two important consequences when the limit in (\ref{weak}) is Gaussian: firstly, as the linear derivatives are (essentially) the evaluation at an appropriate point, by the extended Delta method (see Proposition \ref{Proposition.delta}), the asymptotic distribution of the statistic in (\ref{problem}) is normal; secondly,
the standard bootstrap for (\ref{problem}) is consistent if and only if the underlying map $\phi$ is fully Hadamard differentiable (see \citet[Theorem 3.1]{Fang-Santos}).

\subsection{Totally bounded metric spaces}\label{Subsection.Totally.bounded}

If $\mathbb{Q}$ is a tight Borel measurable map into $ \ell^\infty(\X)$ as in (\ref{weak}), then there is a pseudo-metric on $\X$ such that the sample paths of $\mathbb{Q}$ are uniformly continuous and $\X$ is totally bounded (see \citet[Lemma 1.5.9]{van der Vaart-Wellner}).
For statistical applications it is therefore important to determine conditions under which the derivatives in (\ref{derivatives}) have similar expressions as those in Corollary \ref{Corollary.Compact.2} when the underlying space is totally bounded.

We recall that if $(\X,d)$ is a totally bounded metric space, $(\bar \X, d)$ is a compact metric space, where $\bar \X$ is the completion of $\X$ with respect to $d$. Further, the space $\mathcal{C}_u(\X,d)$ of bounded and uniformly continuous functions $f:\X\longrightarrow \R$ is isometric to $\mathcal{C}(\bar\X,d)$. Each $f\in \mathcal{C}_u(\X,d)$ has a unique extension to a function $\bar f \in \mathcal{C}(\bar\X,d)$. For $x\in \bar{\X}\setminus \X$, this extension is defined by $\bar f(x)=\lim_{n\to\infty} f(x_n)$, with $\{x_n\}\subset \X$ such that $x_n \to x$ (in fact, Cauchy-continuity is enough to check that $\bar f$ is well-defined, but uniform continuity suffices for our purposes).

In this setting, it is straightforward to check that Corollary \ref{Corollary.Compact} still holds if we substitute the sets $A_0(\cdot)$ and  $B_0(\cdot)$ by
\begin{equation}\label{A0barB0bar}
\begin{split}
\bar A_0(h):=& \Big\{ x\in\bar\X: \text{ there exists } \{x_n\}\subset \X \text{ with } x_n\to x\text{ and }h(x_n)\to\sup_{ \X}  h  \Big\},\\
\bar B_0(h):=& \Big\{ x\in\bar\X: \text{ there exists } \{x_n\}\subset \X \text{ with } x_n\to x\text{ and }h(x_n)\to\inf_{ \X}  h  \Big\},
\end{split}
\end{equation}
for $h\in \ell^\infty(\X)$. In particular, the following corollary, important for statistical applications in which $\X$ is a class of functions (see Section \ref{Setion.Maximum.mean.discrepancies}), holds.

\begin{corollary}\label{Corollary.Totally.Compact}

Let $(\X,d)$ be a totally bounded metric space and let $\delta$, $\sigma$, $\iota$ and $\alpha$ be the maps defined in (\ref{funtionals}).
\begin{enumerate}
\item[(a)] The maps $\sigma$, $\iota$ and $\alpha$ are Hadamard directionally differentiable at $f\in \ell^\infty(\X)$ tangentially to the set $\mathcal{C}_u(\X,d)$ with derivatives, for $g\in\mathcal{C}_u(\X,d)$,
\begin{equation*}
\sigma^\prime_f(g)= \sup_{\bar A_0(f)} \bar g, \quad  \iota^\prime_f(g)= \inf_{\bar B_0(f)}\bar  g\quad\text{and}\quad \alpha^\prime_f(g)=\sup_{\bar A_0(f)} \bar g -\inf_{\bar B_0(f)} \bar g,
\end{equation*}
where $\bar A_0(\cdot)$ and $\bar B_0(\cdot)$ are defined in (\ref{A0barB0bar}).

\item[(b)] If additionally $f\in \mathcal{C}_u(\X,d)\setminus \{0\}$, we have that
\begin{equation*}
\delta^\prime_f(g)=\sup_{\bar M^+ (|f|)} (\bar g \cdot \sgn(\bar f)),\, \sigma^\prime_f(g)= \sup_{\bar M ^+ (f)} \bar g, \, \iota^\prime_f(g)= \inf_{\bar M^-(f)} \bar g,\,  \alpha^\prime_f(g)=\sup_{\bar M^+(f)} \bar g -\inf_{\bar M^-(f)} \bar g,
\end{equation*}
where for $h\in \mathcal{C}_u(\X,d)$,
\begin{equation}\label{M+barM-bar}
\bar M^+(h):= \Big\{ x\in\bar \X:  \bar h(x)=\sup_\X h  \Big\}\quad\text{and}\quad \bar M^-(h):= \Big\{ x\in\bar \X: \bar h(x)=\inf_\X h  \Big\}.
\end{equation}

\end{enumerate}

\end{corollary}

\begin{remark}\label{Remark.Full.Totally.bounded}
Corollary \ref{Corollary-Full-differentiability} still holds if $(\X,d)$ is a totally bounded metric space and we replace $\mathcal{C}(\X,d)$, $A_0(\cdot)$ and $B_0(\cdot)$ with $\mathcal{C}_{u}(\X,d)$, $\bar A_0(\cdot)$ and $\bar B_0(\cdot)$ (defined in (\ref{A0barB0bar})), respectively.

\end{remark}

\subsection{Weakly compact sets}\label{Subsection.Weakly.compact}

The compacteness assumption on $\X$ in Corollaries \ref{Corollary.Compact} and \ref{Corollary.Compact.2} could be too demanding in some infinite-dimensional settings. A simple inspection of the proof of Corollary \ref{Corollary.Compact} shows that a similar result
can be stated when $\X$ is a weakly compact subset of a Banach space by using Eberlein-\v{S}mulian theorem (see \citet[p. 163]{Conway}). In such a case, Corollary \ref{Corollary.Compact} still holds by substituting the sets $A_0(h)$ and $B_0(h)$ in (\ref{A0B0}) and the quantities $h^\filledtriangleup_f(x)$ and $h^\filledtriangledown_f(x)$ in (\ref{h+h-}) respectively by
\begin{equation}\label{A0wB0w}
\begin{split}
A_0^w(h):=& \Big\{ x\in\X: \text{ there exists } \{x_n\}\subset \X \text{ with } x_n\rightharpoonup x\text{ and }h(x_n)\to\sup_\X h  \Big\},\\
B_0^w(h):=& \Big\{ x\in\X: \text{ there exists } \{x_n\}\subset \X \text{ with } x_n\rightharpoonup x\text{ and }h(x_n)\to\inf_\X h  \Big\},
\end{split}
\end{equation}
and
\begin{equation}\label{h+wh-w}
\begin{split}
h^{\filledtriangleup,w}_f(x):=& \sup \Big\{\limsup_{n\to\infty} h(x_n) : x_n\rightharpoonup x\text{ and } f(x_n)\to\sup_\X f   \Big\},\quad x\in A_0^w(f),\\
h^{\filledtriangledown,w}_f(x):=& \inf \Big\{\liminf_{n\to\infty} h(x_n) : x_n\rightharpoonup x\text{ and } f(x_n)\to\inf_\X f   \Big\},\quad x\in B_0^w(f),
\end{split}
\end{equation}
where $x_n\rightharpoonup x$ stands for the weak convergence in the corresponding space. We recall that if $\{x_n\}\subset \mathcal{B}$ with $\mathcal{B}$ a Banach space, $x_n\rightharpoonup x$ means that $\varphi(x_n)\to\varphi(x)$ for all $\varphi\in \mathcal{B}^*$, the topological dual space of $\mathcal{B}$ formed by linear and continuous functionals from $\mathcal{B}$ to $\R$. If $\mathcal{B}=\mathcal{H}$ is a Hilbert space with inner product $\inp{\cdot}{\cdot}$, the weak convergence amounts to $\inp{x_n}{y}\to \inp{x}{y}$, for all $y\in\mathcal{H}$.

In this context, we have analogous results as Corollaries \ref{Corollary.Compact.2} and \ref{Corollary.Totally.Compact} by changing the set of tangency points.
\begin{definition}\label{Definition-prelinear}
If $\X$ is a subset of a vector space, a function $g:\X\longrightarrow\R$ is said to be {\rm prelinear} on $\X$ if $\sum_{i=1}^r \lambda_i g(x_i)=0$ whenever $\sum_{i=1}^r \lambda_i x_i=0$, for $r<\infty$, $\lambda_i\in\R$ and $x_i\in \X$ ($i=1,\dots,r$).
\end{definition}
Every prelinear function $g$ defined on $\X$ admits a unique extension to a linear function on $\text{span}(\X)$, the linear span of $\X$  (see \citet[Lemma 2.30, p. 88]{Dudley-1999}). This extension is given by
\begin{equation}\label{extension}
\tilde g\left(\sum_{i=1}^r \lambda_i x_i\right)=\sum_{i=1}^r \lambda_i g(x_i),\quad \text{with }x_i\in\X\text{ and } \lambda_i\in\R\ (i=1,\dots,r).
\end{equation}

In the following, if $(\X,d)$ is a metric space contained in a vector space, we denote by $\mathcal{C}_{pl}(\X,d)$ the subset of $\mathcal{C}(\X,d)$ formed by prelinear functions on $\X$. Further, if $\mathcal{B}$ is a Banach space with norm $\Vert \cdot \Vert$, $d_{\mathcal{B}}$ stands for the metric on $\mathcal{B}$, i.e., $d_{\mathcal{B}}(x,y)=\Vert x-y\Vert$ ($x,y\in\mathcal{B}$). 

\begin{corollary}\label{Corollary.weakly.compact}
Let $\mathcal{B}$ be a Banach space and let $\delta$, $\sigma$, $\iota$ and $\alpha$ be the maps in (\ref{funtionals}). Let us assume that the set $\X\subset \mathcal{B}$ satisfies the following two conditions:
\begin{enumerate}
\item[(i)] $\X$ is a weakly compact subset of $\mathcal{B}$.
\item[(ii)] For each $g\in \mathcal{C}_{pl}(\X,d_{\mathcal{B}})$, its linear extension $\tilde g$ in (\ref{extension}) is continuous on $\text{\rm span}(\X)$.
\end{enumerate}
Then, the maps $\sigma$, $\iota$ and $\alpha$ are Hadamard directionally differentiable at $f\in \ell^\infty(\X)$ tangentially to $\mathcal{C}_{pl}(\X,d_{\mathcal{B}})$ with derivatives, for $g\in\mathcal{C}_{pl}(\X,d_{\mathcal{B}})$,
\begin{equation*}
\sigma^\prime_f(g)= \sup_{A_0^w(f)} g, \quad  \iota^\prime_f(g)= \inf_{B_0^w(f)} g\quad\text{and}\quad \alpha^\prime_f(g)=\sup_{A_0^w(f)} g -\inf_{B_0^w(f)} g,
\end{equation*}
where $A_0^w(\cdot)$ and $B_0^w(\cdot)$ are defined in (\ref{A0wB0w}).

If additionally $f\in \mathcal{C}_{pl}(\X,d_{\mathcal{B}})\setminus \{0\}$, then the derivatives of $\delta$, $\sigma$, $\iota$ and $\alpha$ are as in (\ref{derivative-continuous}).
\end{corollary}

\begin{proof}
As in the previous proofs, we only discuss the map $\sigma$. Let us consider $x\in A_0^w(f)$ (defined in (\ref{A0wB0w})) and $g\in \mathcal{C}_{pl}(\X,d_{\mathcal{B}})$. We consider a sequence $\{x_n\}\subset \X$ such that $x_n \rightharpoonup x$ and $f(x_n)\to \sup_\X f$ (the existence of such a sequence is guaranteed by condition (i) and (\ref{sup-sequences})). Condition (ii) and Hahn-Banach theorem imply that there exists a linear and continuous map, say $\bar g$, defined on $\mathcal{B}$ such that $\bar g= \tilde g$ on $\text{\rm span}(\X)$, and hence $\bar g=g$ on $\X$. As $\bar g\in \mathcal{B}^*$ and $x_{n}\rightharpoonup x$, we conclude that $\lim_{n\to\infty} g(x_{n})=g(x)$. This shows that $g^{\filledtriangleup,w}_f(x)=g(x)$, with $g^{\filledtriangleup,w}_f(x)$ defined as in (\ref{h+wh-w}), and the conclusion follows from the observation at the beginning of this section.

Finally, if $f\in \mathcal{C}_{pl}(\X,d_{\mathcal{B}})$, the same argument used before shows that $A_0^w(f)=M^+(f)$, where the set $M^+(\cdot)$ is defined in (\ref{M+M-}).
\end{proof}

We observe that hypothesis (i) in the previous corollary is essential to extract a weakly convergent subsequence in $\X$. We also observe that condition (ii) cannot be dropped as, in general, the linear extension $\tilde g$
of a function $g\in \mathcal{C}_{pl}(\X,d_{\mathcal{B}})$ is not necessarily continuous in $\text{span}(\X)$ as the following example shows: Let $\mathcal{B}$ be an infinite-dimensional Banach space with norm $\Vert \cdot \Vert$. We consider $\X=\{ x_n\}_{n=0}^\infty\subset \mathcal{B}$, where $x_0=0$ and $\{x_n\}_{n=1}^\infty$ is a linearly independent subset of $\mathcal{B}$ such that $\Vert x_n\Vert=1/n$ ($n\in\mathbb{N}$). It is easy to check that the function defined by $g(0)=0$ and $g(x_n)=1/\sqrt{n}$ ($n\in\mathbb{N}$) belongs to $\mathcal{C}_{pl}(\X,d_{\mathcal{B}})$, but its linear extension $\tilde g$ is \textit{not} continuous because it is not bounded on the unit sphere since $\tilde g(x_n/\Vert x_n\Vert)=\sqrt{n}$ ($n\in\mathbb{N}$).

The following proposition provides easy to check conditions guaranteeing that Corollary \ref{Corollary.weakly.compact} (ii) is fulfilled.

\begin{proposition}\label{Proposition.extension}
Let $\mathcal{B}$ be a Banach space with norm $\Vert \cdot \Vert$ and $\X\subset \mathcal{B}$.  Let us assume that one of the following two conditions is satisfied:
\begin{enumerate}
\item[(a)] There exists $x\in\X$ and $\delta>0$ such that $B(x,\delta):=\{  y\in\mathcal{B}: \Vert y-x \Vert \le \delta \}\subset \X$.

\item[(b)] $\mathcal{B}$ is a Hilbert space and there exists  $\{x_i\}_{i\in I}\subset  \X$, where $I$ is an arbitrary index set such that $\text{\rm span}(\X)=\text{\rm span}(\{x_i\}_{i\in I})$,  $\{x_i\}_{i\in I}$ are pairwise orthogonal and $c:=\inf_{i\in I} \Vert x_i\Vert>0$.
\end{enumerate}
Then, for each $g\in \mathcal{C}_{pl}(\X,d_{\mathcal{B}})$, its linear extension  $\tilde g$ in (\ref{extension}) is continuous on $\text{\rm span}(\X)$.
\end{proposition}
\begin{proof}
Let us assume that (a) holds. As $g\in \mathcal{C}(\X,d_{\mathcal{B}})$, the condition $B(x,\delta)\subset \X$ ensures that $\tilde g$ in (\ref{extension}) is continuous at $x$, and, by linearity, continuous on $\text{\rm span}(\X)$.

Assume now that (b) is satisfied. For $x\in \text{\rm span}(\X)$, we can write $x=\sum_{i=1}^r \lambda_i x_i$, with $\lambda_i\in\R$ $(i=1,\dots,r)$. Taking into account that $\Vert x\Vert = \sum_{i=1}^r   |\lambda_i| \Vert x_i\Vert \ge  c \sum_{i=1}^r   |\lambda_i|  $, we finally obtain that
\begin{equation*}
|\tilde g(x)|\le   \Vert g\Vert_\infty  \sum_{i=1}^r |\lambda_i| \le \Vert g\Vert_\infty \Vert x \Vert /c.
\end{equation*}
The previous inequalities show that $\tilde g$ is continuous on $\text{\rm span}(\X)$ and the proof is complete.
\end{proof}

Closed bounded convex subsets of a reflexive Banach space are weakly compact (see \citet[Corollary 3.22]{Brezis}). Therefore, the hypotheses of Corollary \ref{Corollary.weakly.compact} are general enough to include many infinite-dimensional sets. Thanks to Proposition \ref{Proposition.extension}, an important example covered by Corollary \ref{Corollary.weakly.compact} is when $\X$ is the closed unit ball of a reflexive Banach space, and, in particular, the closed unit ball of a Hilbert space. On the other hand, working with prelinear functions could seem to be too restrictive. However, we point out that if $\P$ is a probability measure and a set $\X$ is $\P$-pre-Gaussian (\citet[Definition 3.7.26, p. 251]{Gine-Nickl}), there is a version of the $\P$-bridge whose sample paths are prelinear (see \citet[Theorem 3.7.28, p. 252]{Gine-Nickl}). Such a version is usually called \textit{suitable}.

\begin{remark}\label{Remark.Full.Weakly}
Corollary \ref{Corollary-Full-differentiability} still holds with the obvious modifications if $\X$ is in the conditions of Corollary \ref{Corollary.weakly.compact}. It is enough to replace convergence with weak convergence and $\mathcal{C}(\X,d)$, $A_0(\cdot)$ and $B_0(\cdot)$ with $\mathcal{C}_{pl}(\X,d_\mathcal{B})$, $A_0^w(\cdot)$ and $B_0^w(\cdot)$, respectively.

\end{remark}

\subsection{The case $\X=\bar \R^d$ and the Skorohod space $\D(\bar\R^d)$}\label{Subection.Skorohod.Space}


Throughout this section $\X=\bar \R^d$ ($d\ge 1$) endowed with $d_e$, the metric corresponding to the Euclidean norm on $[0,1]^d$ through a given homeomorphism. Hence, $(\bar \R^d,d_e)$ is a compact metric space and we can apply Corollaries \ref{Corollary.Compact} and \ref{Corollary.Compact.2} in Section \ref{Subsection.Compact.Spaces}.

Many important stochastic processes take values in the one-dimensional Skorohod space, $\mathcal{D}(\bar\R)$, consisting of all the \textit{c\`{a}dl\`{a}g} functions, that is, right-continuous functions having limit from the left at every point. This space provides a natural and convenient setting to analyze the behaviour of processes with unidimensional time parameter and jumps in their paths such as Poisson processes, L\'evy processes, empirical processes or discretizations of stochastic processes, among others. Skorohod-type spaces are usually equipped with different norms to make them separable. However, we are only interested in a multidimensional extension of the Skorohod space viewed as a subset of $\ell^\infty(\bar  \R^d)$ with the supremum norm.
The final aim of this section is providing alternative expressions for the directional derivatives in (\ref{derivatives-2}) when the involved functions belong to the $d$-dimensional Skorohod space.

The $d$-dimensional Skorohod space, introduced in \cite{Neuhaus} (see also \cite{Bickel-Wichura}) and more recently considered in \cite{Seijo-Sen}), is usually defined in compact rectangles of $\R^d$. We will firstly extend this space to functions defined in $\bar\R^d$.

For $v\in\{-1,1\}$ and $x\in\bar\R$, let
\begin{equation*}
I_v(x):=
\begin{cases}
\emptyset, & \text{if }v=-1,\, x=-\infty,\\
[-\infty,x), & \text{if }v=-1,\, x\in(-\infty,+\infty],\\
(x,+\infty], & \text{if }v=+1,\, x\in [-\infty,+\infty),\\
\emptyset, & \text{if }v=+1,\, x=+\infty,\\
\end{cases}
\end{equation*}
and
\begin{equation*}
\tilde I_v(x):=
\begin{cases}
[-\infty,x), & \text{if }v=-1,\, x<\infty,\\
\bar\R, & \text{if }v=-1,\, x=+\infty,\\
\emptyset, & \text{if }v=+1,\, x=+\infty,\\
[x,+\infty], & \text{if }v=+1,\, x<\infty.\\
\end{cases}
\end{equation*}
We consider $\mathcal{V}:=\{-1,1\}^d$ the set of $2^d$ vertices of $[-1,1]^d$. For $\mathbf{v}=(v_1,\dots,v_d)\in \mathcal{V}$ and $\mathbf{x}=(x_1,\dots,x_d)\in \bar\R^d$, we define the $\mathbf{v}$-quadrants of $\mathbf{x}$ by
\begin{equation*}
Q_{\mathbf{v}}(\mathbf{x}):=I_{v_1}(x_1)\times\cdots\times I_{v_d}(x_d)\quad\text{and}\quad \tilde Q_{\mathbf{v}}(\mathbf{x}):=\tilde I_{v_1}(x_1)\times\cdots\times \tilde I_{v_d}(x_d).
\end{equation*}
Observe that $Q_{\mathbf{v}}(\mathbf{x})\subset \tilde Q_{\mathbf{v}}(\mathbf{x})$, $\tilde Q_{\mathbf{v}}(\mathbf{x})\cap \tilde Q_{\mathbf{v}'}(\mathbf{x})=\emptyset$ whenever $\mathbf{v}$, $\mathbf{v}'\in\mathcal{V}$ with $\mathbf{v}\ne \mathbf{v}'$, and $\cup_{\mathbf{v}\in\mathcal{V}} \tilde Q_{\mathbf{v}}(\mathbf{x})=\bar\R^d$, for all $\mathbf{x}\in \bar\R^d$. Additionally, for each $\mathbf{x}\in \bar\R^d$, there exists a unique $\mathbf{v}_\mathbf{x}\in\mathcal{V}$ such that $\mathbf{x}\in \tilde Q_{\mathbf{v}_\mathbf{x}}(\mathbf{x})$. For instance, if $\mathbf{x}\in\R^d$, we have that $\mathbf{v}_\mathbf{x}=\mathbf{1}$, where $\mathbf{1}:=(1,\dots,1)$.

With the previous concepts we can define the quadrant limits. Let us consider a function $f:\bar\R^d\longrightarrow\R$, $\mathbf{v}\in\mathcal{V}$ and $\mathbf{x}\in \bar\R^d$. We say that $l\in\R$ is the $\mathbf{v}$-limit of $f$ at $\mathbf{x}$ if $Q_{\mathbf{v}}(\mathbf{x})\ne \emptyset$ and for every sequence $\{\mathbf{x}_n\}\subset Q_{\mathbf{v}}(\mathbf{x})$ such that $\mathbf{x}_n\to \mathbf{x}$, we have that $f(\mathbf{x}_n)\to l$. In such a case, we  denote $l\equiv f_\mathbf{v}(\mathbf{x})$. Additionally, it is said that $f$ is continuous from above at $\mathbf{x}\in \bar\R^d$ if $f_{\mathbf{v}_\mathbf{x}}(\mathbf{x})$ exists and $f_{\mathbf{v}_\mathbf{x}}(\mathbf{x})=f(\mathbf{x})$. We say that $f$ is continuous from above  if it is continuous from above at every $\mathbf{x}\in\bar\R^d$.

\begin{definition}
The {\rm Skorohod space} on $\bar\R^d$, denoted by $\D(\bar\R^d)$, is the collection of all continuous from above real functions $f$ defined in $\bar\R^d$ for which the $\mathbf{v}$-limit of $f$ exists for every $\mathbf{v}\in\mathcal{V}$ and $\mathbf{x}\in \bar\R^d$ such that $Q_{\mathbf{v}}(\mathbf{x})\ne\emptyset$.
\end{definition}

When $d=1$, $\D(\bar\R)$ is usual Skorohod space on $\bar \R$. The properties of the multidimensional Skorohod space in $[0,1]^d$ shown in \cite{Neuhaus} can be extended with no difficulty to $\D(\bar\R^d)$.
For instance, the elements in $\D(\bar\R^d)$ belong to $\D(\bar\R)$ in each coordinate, have at most countably many discontinuities and all of them are of the ``first class". The fact that $\D(\bar\R^d)\subset \ell^\infty(\bar\R^d)$ follows from \citet[Corollary 1.6]{Neuhaus} by noting that functions in $\D(\bar\R^d)$ have finite quadrant limits at infinity points.

\begin{remark}\label{Remark.quadrants}
We observe that if $f\in\D(\bar\R^d)$ and $\{\mathbf{x}_n\}\subset \tilde Q_{\mathbf{v}}(\mathbf{x})$ such that $\mathbf{x}_n\to \mathbf{x}$, then $f(\mathbf{x}_n)\to f_{\mathbf{v}}(\mathbf{x})$. This follows from the fact that
\begin{equation*}
\tilde Q_\mathbf{v}(\mathbf{x})=\left\{   \mathbf{y}\in\bar\R^d :  \mathbf{y}\in \overline{Q_{\mathbf{v}_\mathbf{y}}(\mathbf{y})\cap Q_\mathbf{v}  (\mathbf{x})}    \right\},
\end{equation*}
where $\bar A$ denotes the closure of the set $A$. In other words, the functions in $\D(\bar\R^d)$ have quadrant limits in $\tilde Q_{\mathbf{v}}(\mathbf{x})$.
\end{remark}



We are now in position to see how the derivatives in (\ref{derivatives-2}) look like when $\X=\bar \R^d$ and the functions on which they act belong to $\D(\bar\R^d)$.

\begin{corollary}\label{Corollary.Skorohod}

For any  $f\in \D(\bar\R^d)\setminus \{0\}$, the maps $\delta$, $\sigma$, $\iota$ and $\alpha$ in (\ref{funtionals}) are Hadamard directionally differentiable at $f$ tangentially to $\D(\bar\R^d)$. For $g\in\D(\bar\R^d)$, their derivatives are given by
\begin{equation}\label{derivatives-Skorohod}
\begin{aligned}
\delta^\prime_f(g)&= \max_{\mathbf{v}\in\mathcal{V}}\sup_{M_\mathbf{v}^+(|f|)}\left( g_\mathbf{v}\cdot \sgn(f_\mathbf{v})\right),\quad   &  \sigma^\prime_f(g)&=\max_{\mathbf{v}\in\mathcal{V}}\sup_{M^+_\mathbf{v}(f)} g_\mathbf{v},\\[2 mm]
\iota^\prime_f(g) &= \min_{\mathbf{v}\in\mathcal{V}}\inf_{M^-_\mathbf{v}(f)} g_\mathbf{v},\quad   &   \alpha^\prime_f(g)&=\max_{\mathbf{v}\in\mathcal{V}}\sup_{M^+_\mathbf{v}(f)} g_\mathbf{v}- \min_{\mathbf{v}\in\mathcal{V}}\inf_{M^-_\mathbf{v}(f)} g_\mathbf{v},
\end{aligned}
\end{equation}
where for $h\in\D(\bar\R^d)$,
\begin{equation}\label{Mf+}
\begin{split}
M_\mathbf{v}^+(h)&:=\left\{ \mathbf{x}\in\  \bar\R^d : Q_{\mathbf{v}}(\mathbf{x})\ne \emptyset \text{ and }h_\mathbf{v}(\mathbf{x})=\sup h \right\},\\
M_\mathbf{v}^-(h)&:=\left\{ \mathbf{x}\in\  \bar\R^d : Q_{\mathbf{v}}(\mathbf{x})\ne \emptyset \text{ and }h_\mathbf{v}(\mathbf{x})=\inf h \right\}.
\end{split}
\end{equation}

\end{corollary}
\begin{proof}
This corollary can be proved as Corollary \ref{Corollary.Compact} by taking into account Remark \ref{Remark.quadrants} and the following fact: As the number of non-empty quadrants of each point in $\bar\R^d$ is finite, each sequence converging to a point $\mathbf{x}\in \bar\R^d$ has a subsequence contained in $\tilde Q_{\mathbf{v}}(\mathbf{x})$, for some $\mathbf{v}\in \mathcal{V}$. In particular, for every $h\in\D(\bar\R^d)$, it holds that $A_0(h)=\cup_{\mathbf{v}\in\mathcal{V}} M_\mathbf{v}^+(h)$ and $B_0(h)=\cup_{\mathbf{v}\in\mathcal{V}} M_\mathbf{v}^-(h)$, where $A_0(h)$ and $B_0(h)$ are defined in (\ref{A0B0}).
\end{proof}

The sets $M_\mathbf{v}^+(h)$  (respectively, $M_\mathbf{v}^-(h)$) in (\ref{Mf+}) might coincide for different $\mathbf{v}\in\mathcal{V}$. For instance, when $f$ is continuous,  $M_\mathbf{v}^+(|f|)=M^+(|f|)$, $M_\mathbf{v}^+(f)=M^+(f)$, and $M_\mathbf{v}^-(f)=M^-(f)$, for all $\mathbf{v}\in\mathcal{V}$, where $M^+(\cdot)$ and $M^-(\cdot)$ are defined in (\ref{M+M-}).

We emphasize that $g_\mathbf{v}\equiv g$, for all $\mathbf{v}\in\mathcal{V}$, whenever $g\in \mathcal{C}(\bar\R^d,d_e)$. The following corollary is important for applications because many stochastic processes that commonly appear as weak limits of other processes have continuous paths a.s.

\begin{corollary}\label{Corollary.Skorohod-2}
For any $f\in \D(\bar\R^d)\setminus \{0\}$, the maps $\delta$, $\sigma$, $\iota$ and $\alpha$ in (\ref{funtionals}) are Hadamard directionally differentiable at $f$ tangentially to $\mathcal{C}(\bar\R^d,d_e)$. For $g\in\mathcal{C}(\bar\R^d,d_e)$, their derivatives are given by
\begin{equation*}
\begin{aligned}
\delta^\prime_f(g) & =  \max_{\mathbf{v}\in\mathcal{V}}\sup_{M_\mathbf{v}^+(|f|)}\left( g \cdot \sgn(f_\mathbf{v})\right) ,\quad   &  \sigma^\prime_f(g) &= \max_{\mathbf{v}\in\mathcal{V}}\sup_{M^+_\mathbf{v}(f)} g,\\[2 mm]
\iota^\prime_f(g) &=  \min_{\mathbf{v}\in\mathcal{V}}\inf_{M^-_\mathbf{v}(f)} g,\quad   &   \alpha^\prime_f(g) &= \max_{\mathbf{v}\in\mathcal{V}}\sup_{M^+_\mathbf{v}(f)} g- \min_{\mathbf{v}\in\mathcal{V}}\inf_{M^-_\mathbf{v}(f)}g,
\end{aligned}
\end{equation*}
with $M_\mathbf{v}^+(\cdot)$ and $M_\mathbf{v}^-(\cdot)$ defined in (\ref{Mf+}).

If additionally $f\in\mathcal{C}(\bar\R^d,d_e)$, the derivatives are as in (\ref{derivative-continuous}).
\end{corollary}

\subsection{Statistical applications}\label{Subsection.Statistical.Applications}

In a wide variety of situations Theorem \ref{Theorem.differentiability} and its subsequent corollaries, joint with the extended Delta method in Proposition \ref{Proposition.delta}, provide the right framework to obtain a number of significant examples in which the asymptotic distribution of a statistic of interest can be determined with ease. The combination of these results is summarized in the following theorem.

\begin{theorem}\label{Theorem.main}
Let $q\in \ell^\infty(\X)\setminus \{ 0 \}$ and assume that there exists $\mathbb{Q}_n$ taking values in $\ell^\infty(\X)$ a.s. such that $r_n(\mathbb{Q}_n-q)\rightsquigarrow \mathbb{Q}$, for a sequence of real numbers satisfying that $r_n\to\infty$ and a Borel random element $\mathbb{Q}$ in $\ell^\infty(\X)$. Then, for $\phi\in\{\delta,\sigma,\iota, \alpha\}$ in (\ref{funtionals}), we have that
\begin{equation}
r_n(\phi(\mathbb{Q}_n)-\phi(q))\cd \phi_q^\prime(\mathbb{Q}),
\end{equation}
where the derivatives $\phi_q^\prime$ are given in (\ref{derivatives}). Moreover, we have that $r_n(\phi(\mathbb{Q}_n)-\phi(q))=\phi_q^\prime(r_n(\mathbb{Q}_n-q))+o_\P(1)$.
\end{theorem}

Theorem \ref{Theorem.main} is still valid for the maps $\sigma$, $\iota$ and $\alpha$ when $q\equiv 0$ as $\sigma_0^\prime(g)\equiv\sup_{\X} g$, $\iota_0^\prime(g)\equiv\inf_{\X} g$ and $\alpha_0^\prime(g)\equiv\amp_{\X}(g)$ are continuous maps. Further, for those $q\in\ell^\infty(\X)$ such that $\phi_q^\prime$ is linear, i.e., $\phi$ is fully Hadamard differentiable at $q$ (see Corollary \ref{Corollary-Full-differentiability} and Remarks \ref{Remark.Full.Totally.bounded} and \ref{Remark.Full.Weakly}), and when $\mathbb{Q}$ is Gaussian, we conclude that $\phi_q^\prime(\mathbb{Q})$ is normally distributed.

In this setting, despite $\phi\in\{\delta,\sigma,\iota, \alpha\}$ in (\ref{funtionals}) not being uniformly Hadamard diffe\-ren\-tiable (see Example \ref{Example-not-uniform}), if we know that $q_n\to q$ in $\ell^\infty(\X)$ and $r_n(\mathbb{Q}_n-q_n)\rightsquigarrow \mathbb{Q}$, we can still conclude that $r_n(\phi(\mathbb{Q}_n)-\phi(q_n))\cd \phi_q^\prime(\mathbb{Q})$ if we assume: (i) $\phi$ is fully Hadamard diffe\-ren\-tiable; (ii) the sequence $r_n(q_n-q)$ is relatively compact. Further, if (i) is not satisfied and $r_n(q_n-q)\to h$, we still have that $r_n(\phi(\mathbb{Q}_n)-\phi(q_n))\to_d  \phi_q^\prime(\mathbb{Q}+h) - \phi_q^\prime(h)$. See \citet[p. 375]{van der Vaart-Wellner} for details.

In what follows we will apply Theorem \ref{Theorem.main} in different contexts to obtain the asymptotic distribution of several statistics.


\section{Distribution functions}\label{Section.Distribution.Functions}

Let $\mathbf{X}$ and $\mathbf{Y}$ be two non-degenerate random vectors ta\-king values on $\R^d$ ($d\ge 1$) with joint cumulative distribution functions $F(\x):=\P(\mathbf{X}\le \x)$ and $G(\x):=\P(\mathbf{Y}\le \x)$, $\x\in\R^d$, where `$\le$' stands for the coordinatewise order in $\R^d$. The goal in this section is to estimate $\phi(F-G)$, where $\phi\in\{\delta, \sigma, \alpha \}$ are defined in (\ref{funtionals}).

\medskip

\noindent \textbf{\textsl{One-sample case:}} In this situation we have at our disposal a random sample $\mathbf{X}_1,\dots,\mathbf{X}_n$ from $\mathbf{X}$. We estimate $F-G$ with $F_n-G$, where $F_n$ is the empirical distribution function of the observed sample, that is,
\begin{equation*}
{F}_n(\x):=\frac{1}{n}\sum_{i=1}^n 1_{\{\mathbf{X}_i\le \x\}}, \quad \x\in\R^d,
\end{equation*}
and $1_A$ stands for the indicator function of the set $A$.

The problem consists in finding the behaviour, as $n\to\infty$, of
\begin{equation}\label{Dn}
\begin{split}
D_n(\delta)&= \sqrt{n}\, \left(\Vert{F}_n-G\Vert_\infty-\Vert F-G\Vert_\infty\right),\\
D_n(\sigma)&=\sqrt{n}\,\left(   \sup ({F}_n-G) - \sup (F-G) \right)\\
D_n(\alpha)&=\sqrt{n}\, \left(   \amp({F}_n-G) - \amp(F-G) \right).
\end{split}
\end{equation}
When $F\not=G$, the asymptotic distribution of the statistics $D_n(\delta)$, $D_n(\sigma)$ and $D_n(\alpha)$ in (\ref{Dn}) can be viewed as the limit under the alternative hypothesis of the corresponding two-sided and one-sided Kolmogorov-Smirnov test statistics and Kuiper statistic, respectively.

In this example, for $\phi\in\{\delta, \sigma, \alpha \}$, the statistics in (\ref{Dn}) are $D_n(\phi)\equiv D_\phi (q,\mathbb{Q}_n,r_n)$ in (\ref{problem}) with $q=F-G$, $\mathbb{Q}_n=F_n-G$, and $r_n=\sqrt{n}$.
The underlying normalized process, i.e., $r_n(\mathbb{Q}_n-q)$, is nothing but the multivariate \emph{empirical process} (indexed by points),
\begin{equation}\label{empirical process}
\Emp_{n,F}(\x):=\sqrt{n}({F}_n(\x)-F(\x)),\quad n\in \mathbb{N},\quad \x\in \R^d.
\end{equation}
When there is no confusion with respect to the underlying distribution, we simply use the notation $\Emp_{n}$ for the empirical process in (\ref{empirical process}).
As the collection of all indicator functions of lower (hyper)rectangles of $\bar\R^d$, $\{ 1_{(-\infty, x_1]\times\cdots\times (-\infty, x_d]} : (x_1,\dots,x_d)\in\bar\R^d \}$, is Donsker (see \citet[Example 2.1.3, p. 82]{van der Vaart-Wellner}), the empirical process converges in law in $\ell^\infty(\bar\R^d)$. The weak limit of $\Emp_n$, denoted in the following by $\B_F$, is a \textit{$F$-Brownian bridge}, that is, a centered Gaussian process with covariance function $\E (\B_F(\x)\B_F(\mathbf{y}))=F(\x\wedge\mathbf{y})-F(\x)F(\mathbf{y})$. (Here $\x\wedge\mathbf{y}\equiv(x_1\wedge y_1,\dots,x_d\wedge y_d)$ if $\x=(x_1,\dots,x_d)$ and $\mathbf{y}=(y_1,\dots,y_d)$.) If $d=1$, the assertion ``$\Emp_n \rightsquigarrow \B_F$ in $\ell^\infty(\bar\R)$" is nothing but the celebrated Donsker's theorem (Kolmogorov-Doob-Donsker-Dudley central limit theorem). In such a case, $\B_F=\B\circ F$, where $\B$ is a standard Brownian bridge on $[0,1]$. When $d\ge 2$, $\B_F$ is also called a tied-down or pinned Brownian sheet based on the measure with distribution function $F$.





In this particular case we have that $F-G\in \D(\bar\R^d)$, $\Emp_n\in\D(\bar\R^d)$ a.s., and $\Emp_n\rightsquigarrow \B_F$ in $\ell^\infty(\bar\R^d)$. Therefore, as a direct consequence of Theorem \ref{Theorem.main} and Corollary \ref{Corollary.Skorohod} we obtain the following result.

\begin{proposition}\label{Proposition.Empirical}
Assume that $F\ne G$ and let $\B_F$ be an $F$-Brownian bridge. For $\phi\in\{\delta, \sigma, \alpha \}$, we consider the statistics $D_n(\phi)$ defined in (\ref{Dn}). We have that $D_n(\phi)\cd \phi_{F-G}^\prime(\B_F)$, where the derivatives $\phi_{F-G}^\prime$ are given as in (\ref{derivatives-Skorohod}).
\end{proposition}

When $d=1$, Proposition \ref{Proposition.Empirical} improves \citet[Theorems 1, 2 and 3]{Raghavachari} as here $F$ and $G$ are not assumed to be continuous. 
If $F$ is continuous, then $\B_F\in \mathcal{C}(\bar\R^d,d_e)$ a.s., and the limiting distributions in Proposition \ref{Proposition.Empirical} have simpler expressions (see (\ref{derivative-continuous})). The following corollary provides a multidimensional extension of the results in \cite{Raghavachari}.

\begin{corollary}
In the conditions of Proposition \ref{Proposition.Empirical}, let us further assume that $F, G\in\mathcal{C}(\bar\R^d,d_e)$ and we consider the sets $M^+(\cdot)$ and $M^-(\cdot)$  defined in (\ref{M+M-}). We have that:
\begin{itemize}
\item[(i)] $\displaystyle D_n(\delta)   \cd \sup_{M^+(|F-G|)}\left( \B_F \cdot \sgn(F-G)\right)$;

\item[(ii)] $\displaystyle D_n(\sigma)  \cd  \sup_{M^+(F-G)} \B_F$;

\item[(iii)] $\displaystyle D_n(\alpha)  \cd \sup_{M^+(F-G)} \B_F -  \inf_{M^-(F-G)} \B_F$.
\end{itemize}
\end{corollary}

\begin{remark}
In the setting of the previous corollary, when $M^+(|F-G|)$ (respectively, $M^+(F-G)$, and $M^+(F-G)$ and $M^-(F-G)$) contains only one point, the mapping $\delta$ (respectively, $\sigma$ and $\alpha$) is fully Hadamard differentiable at $F-G$ (see Corollary \ref{Corollary-Full-differentiability}). In particular, the asymptotic distribution of $D_n(\delta)$ (respectively, $D_n(\sigma)$ and $D_n(\alpha)$) is a zero mean Gaussian distribution. The asymptotic variance can be directly computed from the covariances of $\B_F$.
\end{remark}


\begin{remark}\label{Remark.Dette}
In  \citet[Theorem 6.1]{Dette-Kokot-Aue}, the authors obtained a similar version of the results in \cite{Raghavachari} for convergence of suprema of non-centered processes indexed by directed sets. Using the results in Section \ref{Subsection.Compact.Spaces} we can state the following slightly more general result: Let $(T,d)$ be a compact metric space and $\mu \in \mathcal{C}(T,d)\setminus\{0\}$. Let $\{X_a : a\in A\}$ be a net of random variables taking values in $\ell^\infty(T)$ and $r:A\longrightarrow [0,\infty)$ satisfying that $\lim_{a} r_a = \infty$ (with $r_a=r(a)$). Assume that $Z_a:=r_a(X_a-\mu)\cd Z$ in $\ell^\infty(T)$, where $Z$ is a Gaussian random variable with paths in $\mathcal{C}(T,d)$ a.s., then
$$D_a(\delta)=r_a ( \Vert X_a\Vert_\infty-\Vert \mu\Vert_\infty)\cd \delta^\prime_{\mu}(Z)=\sup_{M^+(|\mu|)} \sgn(\mu) Z.$$
It is worth noting that we can drop the assumption on the normalizing sequence $r_a$ in \citet[Theorem 6.1]{Dette-Kokot-Aue}. A similar result can be provided when $(T,d)$ is a totally bounded metric space by using the results in Section \ref{Subsection.Totally.bounded} (see Corollary \ref{Corollary.Totally.Compact} (b)).
\end{remark}

\begin{remark}
The results in the paper can be used to make inferences on the quantity $\delta(F-F_0)=\Vert F-F_0 \Vert_\infty$, where $F_0$ is a fixed and known distribution function. It should be taken into account that in general the corresponding limiting distribution, $\delta^\prime_{F-F_0}(\B_F)$, cannot be approximated by a standard bootstrap approach. It is known that the standard bootstrap fails when the mapping is not fully Hadamard differentiable and the limit of the underlying process is Gaussian; see \citet[Theorem 3.1]{Fang-Santos}. Observe that the map $\delta$ is fully differentiable at $F-F_0$ if and only if $F-F_0$ is a peaking function (see Corollary \ref{Corollary-Full-differentiability}). Therefore, an alternative approach has to be used if we do not assume this ``peaking condition" on $F-F_0$. In \citet[Theorem 3.2]{Fang-Santos} a method to consistently estimate this type of asymptotic distributions is proposed. The key idea is estimating in a suitable way the directional derivative. A detailed study of these topics is beyond the scope and space limitations of the present paper.
\end{remark}

\medskip

\noindent \textbf{\textsl{Two-sample case:}} Here, two (mutually independent) random samples are available, one of size $n$ from $F$ and another one of size $m$ from $G$. Let ${F}_n$ and ${G}_m$ be the empirical distribution functions of the two samples, respectively, and set $N\equiv\frac{nm}{n+m}$. The two-sided, and one-sided Kolmogorov-Smirnov and Kuiper statistics in the two sample case are given by
\begin{equation}\label{Dnm}
\begin{split}
D_{n,m}(\delta)&:=\sqrt{N}(\Vert{F}_n-{G}_m\Vert_\infty-\Vert F-G\Vert_\infty),\\
D_{n,m}(\sigma)&:=\sqrt{N}\left(\sup({F}_n-{G}_m)-\sup(F-G)\right)\\
D_{n,m}(\alpha)&:=\sqrt{N}(\amp({F}_n-{G}_m)-\amp (F-G)).
\end{split}
\end{equation}

In the general setting specified in (\ref{problem}), this situation corresponds to the case $q=F-G$, $\mathbb{Q}_{n,m}=F_n-G_m$ and $r_{n,m}=\sqrt{N}$. Hence, we have that
\begin{equation*}
r_{n,m}(\mathbb{Q}_{n,m}-q)=\sqrt{\frac{m}{n+m}}\,\Emp_{n,F}-\sqrt{\frac{n}{n+m}}\,\tilde\Emp_{m,G}
\end{equation*}
with $\Emp_{n,F}$ and $\tilde\Emp_{m,G}$ independent empirical processes. We further observe that if the sampling scheme is balanced, that is, $n/(n + m)\to\lambda$, with $0 < \lambda < 1$ as $n,m\to\infty$, then
$r_{n,m}(\mathbb{Q}_{n,m}-q)\rightsquigarrow \sqrt{1-\lambda}\,\B_{F}-\sqrt{\lambda}\,\tilde\B_{G} $ in $\ell^\infty(\bar\R^d)$, where $\B_{F}$ and $\tilde\B_{G}$ are two independent Brownian bridges associated with $F$ and $G$, respectively. Hence, Theorem \ref{Theorem.main} and Corollary \ref{Corollary.Skorohod-2} directly imply the following result which improves and generalizes \citet[Theorems 4 and 5]{Raghavachari}.

\begin{proposition}
Let us consider a sampling scheme such that as $n$, $m\to\infty$, $n/(n + m)\to\lambda$, with $0 < \lambda < 1$ and let  $\B_{F}$ and $\tilde\B_{G}$ be two independent Brownian bridges associated with $F$ and $G$, respectively.  For $\phi\in\{\delta, \sigma, \alpha \}$, we consider the statistics $D_{n,m}(\phi)$ defined in (\ref{Dnm}).  We have that $D_{n,m}(\phi)\cd \phi_{F-G}^\prime(\sqrt{1-\lambda}\,\B_{F}-\sqrt{\lambda}\,\tilde\B_{G})$, where the derivatives $\phi_{F-G}^\prime$ are given in (\ref{derivatives-Skorohod}).
If we further have that $F, G\in\mathcal{C}(\bar\R^d, d_e)$, then the derivatives can be expressed as in (\ref{derivative-continuous}).
\end{proposition}

\section{Copulas}\label{Section-Copula}

In this section, for simplicity, we will assume that the involved distribution functions are continuous. Let us assume that the $d$-dimensional distribution function $F$ has copula $C$ and continuous marginal distribution functions $F_1,\dots,F_d$. In other words, $F(\mathbf{x})=C(F_1(x_1),\dots,F_d(x_d))$, for $\mathbf{x}=(x_1,\dots,x_d)\in\R^d$. Let ${F}_n$ and ${F}_{n,i}$ ($i=1,\dots,d$) be the empirical joint and $i$-th marginal distribution functions of a random sample of size $n$ from $F$.
The \textit{empirical copula} is
\begin{equation}\label{empirical copula}
{C}_n(\mathbf{u}):={F}_n( {F}_{n,1}^{-1}(u_1),\dots, {F}_{n,d}^{-1}(u_d) ),\quad \mathbf{u}:=(u_1,\dots,u_d)\in [0,1]^d,
\end{equation}
where $ {F}_{n,i}^{-1}$ stands for the generalized inverse of ${F}_{n,i}$, i.e., the marginal quantile function of the $i$-th coordinate sample. The \textit{empirical copula process} is defined by
\begin{equation}\label{empirical copula process}
\mathbb{C}_n(\mathbf{u}):=\sqrt{n}({C}_n(\mathbf{u})-C(\mathbf{u})),\quad n\in \mathbb{N},\quad \mathbf{u}\in [0,1]^d.
\end{equation}
Empirical copula processes play the same role for copulas as empirical processes for distribution functions and they have been extensively used in goodness-of-fit testing problems for copulas (see \cite{Fermanian} for an overview about this subject).

Several works have been devoted to discuss the asymptotic behaviour of $\mathbb{C}_n$ in (\ref{empirical copula process}). For instance, in \cite{Segers} (see also the references therein) it is shown that, under certain not very restrictive smoothness assumptions on the underlying copula $C$, $\mathbb{C}_n$ converges weakly in $\ell^\infty([0,1]^d)$. Specifically, let us assume that $C$ satisfies the following regularity condition:

\textbf{Condition 1.} For each $i\in\{1,\dots,d\}$, the $i$-th first order partial derivative of $C$, $\partial_i C$, exists and is continuous on the set $\{ \mathbf{u}=(u_1,\dots.u_d)\in[0,1]^d: 0<u_i<1\}$.

If Condition 1 is satisfied, $\mathbb{C}_n \rightsquigarrow \mathbb{C}$ in $\ell^\infty([0,1]^d)$ (see \citet[Proposition 3.1]{Segers}), where $\mathbb{C}$ is a Gaussian process that can be represented as
\begin{equation}\label{limit copula}
\mathbb{C}(\mathbf{u})=\B_C(\mathbf{u})-\sum_{i=1}^d \partial_i {C}(\mathbf{u}) \B_C^{(i)}(u_i),\quad  \mathbf{u}=(u_1,\dots,u_d)\in [0,1]^d,
\end{equation}
with $\B_C$ a $C$-Brownian bridge (see Section \ref{Section.Distribution.Functions}) and $\B_C^{(i)}(u_i):=\B_C(1,\dots,1,u_i,1,\dots,1)$, the variable $u_i$ appearing at the $i$-th entry.

Using Theorem \ref{Theorem.main} and Corollary \ref{Corollary.Skorohod-2}, we immediately obtain the following result. Though details are omitted, similar results can be stated for the unilateral Kolmogorov-Smirnov and Kuiper statistics and the associated two sample problems. Therefore, we obtain analogous outcomes to those of \cite{Raghavachari} for copulas instead of distribution functions.

\begin{proposition}\label{Proposition Copulas}
Let $C$ be a copula satisfying Condition 1 and let $C_n$ be as in (\ref{empirical copula}). For any copula $D\ne C$, the statistic
\begin{equation*}
T_n(C,D):=\sqrt{n}(\Vert C_n- D\Vert_\infty-\Vert C-D\Vert_\infty)
\end{equation*}
converges in distribution to $\delta^\prime_{C-D}(\mathbb{C})=\sup_{M^+(|C-D|)}\left( \mathbb{C} \cdot \sgn(C-D)\right)$, with $\mathbb{C}$ defined in (\ref{limit copula}) and the set $M^+(\cdot)$ is given in (\ref{M+M-}).
\end{proposition}

For any bivariate copula $C$, we consider the survival copula $\bar C$ defined by
\begin{equation*}
\bar C(u,v):=u+v-1+C(1-u,1-v),\quad  (u,v)\in[0,1]^2.
\end{equation*}
The statistic
\begin{equation} \label{statistics-symmetry}
\bar T_n(C):=\sqrt{n}(\Vert C_n- \bar C_n\Vert_\infty-\Vert C-\bar C\Vert_\infty),
\end{equation}
where $C_n$ is given in (\ref{empirical copula}), has been used in \cite{Genest-Neslehova} to derive a test of radial symmetry for bivariate copulas. The next proposition provides the asymptotic distribution of such statistic.
\begin{proposition}\label{Proposition Copulas-2}
Let $C$ be a bivariate copula satisfying Condition 1 (for $d=2$). The statistic $\bar T_n(C)$ in \eqref{statistics-symmetry}
converges in distribution to
\begin{equation*}
\delta^\prime_{C-\bar C}(\mathbb{C}^*)=\sup_{M^+(|C-\bar C|)}\left( \mathbb{C}^* \cdot \sgn(C-\bar C)\right),
\end{equation*}
where $\mathbb{C}^*(u,v):=\mathbb{C}(u,v)-\mathbb{C}(1-u,1-v)$, $(u,v)\in[0,1]^2$, and $\mathbb{C}$ and the set $M^+(\cdot)$ are defined in (\ref{limit copula}) and (\ref{M+M-}), respectively.
\end{proposition}
\begin{proof}
From Theorem \ref{Theorem.main}, it will suffices to show that
\begin{equation*}
C_n^*:=\sqrt{n} \left( C_n-\bar C_n -(C-\bar C)     \right) \rightsquigarrow \mathbb{C}^* \quad \text{in }\ell^\infty([0,1]^2).
\end{equation*}
Observe that $C_n^*(u,v)= \mathbb{C}_n(u,v)-\mathbb{C}_n(1-u,1-v)$, $(u,v)\in[0,1]^2$, with $\mathbb{C}_n$ being the empirical copula process defined in \eqref{empirical copula process}. Therefore, from Condition 1 together with \citet[Proposition 3.1]{Segers}, and the continuous mapping theorem, we have that $C_n^* \rightsquigarrow \mathbb{C}^*$ in $\ell^\infty([0,1]^2)$ and the proof is complete.
\end{proof}

\section{On a question by Jager and Wellner related to the Berk-Jones statistic}\label{Setion.Jager.Weller}

Let $F_n$ be the empirical distribution function of a sample of size $n$ from a univariate random variable with continuous distribution function $F$. Suppose that we want to test the null hypothesis $\text{H}_0: F=G$ versus the alternative $\text{H}_1: F\ne G$, where $G$ is a fixed (and usually known) continuous distribution function.  \cite{Berk-Jones} (see also \citet[Chapter 26.7]{DasGupta}) introduced the test statistic
\begin{equation}\label{R}
R(F_n,G):=\sup_{x\in\R} K(F_n(x),G(x)),
\end{equation}
where
\begin{equation*}
K(x,y):=x\log\left(\frac{x}{y}\right)+(1-x)\log\left(\frac{1-x}{1-y}\right),
\end{equation*}
for $x \in [0,1]$ and $y \in (0,1)$. (The values of $K(x,y)$ when $x=0$ and $x=1$ are taken by continuity.)

For each $x\in\R$, $nK(F_n(x),G(x))$ is the log-likelihood ratio statistic for testing $\text{H}_0: F(x)=G(x)$ against $\text{H}_1:F(x)\ne G(x)$. Hence, $R(F_n,G)$ in (\ref{R}) is nothing but the supremum of these pointwise likelihood ratio tests statistics. Additionally, $K(x,y)$ is the Kullback-Leibler divergence between two Bernoulli distributions with means $x$ and $y$. Hence, $K(x,y)\ge 0$ with equality if and only if $x=y$. In particular, $R(F_n,G)=\Vert K(F_n,G) \Vert_\infty$.

\cite{Berk-Jones} computed the asymptotic
distribution of (the normalized version of) $R(F_n,F)$, i.e., the distribution of the statistic under the null hypothesis $F=G$. For a detailed proof, see \citet[Theorem 1.1]{Wellner-Koltchinskii} or \citet[Theorem 3.1]{Jager-Wellner-2007}. It holds that
\begin{equation}\label{BJ null}
n R(F_n,F)-d_n\cd Y_4, \quad \text{as }n\to\infty,
\end{equation}
where $\P(Y_4\le x)=\exp(-4\exp(-x))$ for $x\in\R$, i.e., $Y_4$ has double-exponential extreme value distribution, and
\begin{equation*}
d_n:= \log_2 n -\frac{1}{2} \log_3 n - \frac{1}{2} \log (4\pi),
\end{equation*}
with $\log_2 n:= \log (\log n )$ and $\log_3 n:= \log(\log_2 n)$.

In \citet[Question 2, p. 329]{Jager-Wellner-2004}, it was set out the open problem of finding the asymptotic behaviour of the Berk-Jones statistic under the alternative hypothesis. In other words, assuming that $F\ne G$, the question consists in finding conditions on $F$ and $G$ for which the statistic
\begin{equation}\label{Bn}
B_n:=\sqrt{n} \big( R(F_n,G)- R(F,G) \big),
\end{equation}
converges in distribution and, in such a case, identifying its weak limit, where $R(F_n,G)$ is given in (\ref{R}) and $R(F,G):=\sup_{x\in\R} K(F(x),G(x))$.

Here we give a precise answer for the previous question. First, we note that $B_n$ in (\ref{Bn}) has the general form of (\ref{problem}). In other words,
\begin{equation}\label{representation Bn}
B_n=D_\sigma(q=K(F,G),\mathbb{Q}_n=K(F_n,G),r_n=\sqrt{n}),
\end{equation}
where $\sigma$ is defined in (\ref{funtionals}). 
Therefore, from (\ref{representation Bn}) and Theorem \ref{Theorem.main}, to obtain the asymptotic distribution of $B_n$ in (\ref{Bn}) it is enough to find the weak limit of the process $\mathbb{W}_n$ given by
\begin{equation}\label{Wn process}
\mathbb{W}_n:= \sqrt{n}(K(F_n,G)-K(F,G)).
\end{equation}
This result is stated in the following theorem.

\begin{theorem}\label{Theorem BJ}
Let us assume that the function $\log\left(\frac{F\,(1-G)}{G\,(1-F)}\right)$ is monotone around $\pm \infty$ and
\begin{equation*}
\int_\R \log^2\left(        \frac{F(t)(1-G(t))}{G(t)(1-F(t))}    \right)\,\d F(t)<\infty.
\end{equation*}
The process $\mathbb{W}_n$ defined in (\ref{Wn process}) satisfies that $\mathbb{W}_n\rightsquigarrow \mathbb{W}$ in $\ell^\infty(\bar\R)$, where
\begin{equation}\label{W}
\mathbb{W}:=\B_F \log\frac{F(1-G)}{G(1-F)},
\end{equation}
and $\B_F$ is an $F$-Brownian bridge.
\end{theorem}

\begin{proof}
Using Taylor's theorem, we have that
\begin{equation}\label{Taylor}
K(F_n,G)-K(F,G)=(F_n-F) \log\frac{F(1-G)}{G(1-F)}+\frac{1}{2}\frac{(F_n-F)^2}{F_n^*(1-F_n^*)},
\end{equation}
where $F_n^*$ is between $F$ and $F_n$. We set
\begin{equation}\label{Weighted empirical process}
\tilde{\mathbb{W}}_n:= \sqrt{n}(F_n-F)\,  \log\frac{F(1-G)}{G(1-F)}.
\end{equation}
From (\ref{Wn process}) and (\ref{Taylor}), we have that
\begin{equation}\label{Wn-tildeWn}
\Vert \mathbb{W}_n-\tilde{\mathbb{W}}_n\Vert_\infty=\frac{\sqrt{n}}{2}\left\Vert \frac{(F_n-F)^2}{F_n^*(1-F_n^*)}\right\Vert_\infty.
\end{equation}
Now, from (\ref{Wn-tildeWn}) and \citet[equation (2.2)]{Wellner-Koltchinskii} (see also \citet[equation (9)]{Jager-Wellner-2007}, we obtain that
\begin{equation}\label{Wn-tildeWn-2}
\begin{split}
\Vert \mathbb{W}_n-\tilde{\mathbb{W}}_n\Vert_\infty  &  =_{\text{st}}   \sqrt{n} R(F_n,F)\\ 
&=\frac{1}{\sqrt{n}}(nR(F_n,F) - d_n)+ \frac{d_n}{\sqrt{n}},
\end{split}
\end{equation}
where `$=_{\text{st}} $' stands for equality in distribution.
From (\ref{BJ null}) and (\ref{Wn-tildeWn-2}), we conclude that $\Vert \mathbb{W}_n-\tilde{\mathbb{W}}_n\Vert_\infty\cd 0$.
Hence, the processes $\mathbb{W}_n$ and $\tilde{\mathbb{W}}_n$ have the same asymptotic behaviour (see \citet[Theorem 18.10]{van der Vaart}).
Finally, the conclusion follows from \citet[Example 19.12, p. 273]{van der Vaart}.
\end{proof}

\begin{remark}
As it follows from the proof of Theorem \ref{Theorem BJ}, the process $\mathbb{W}_n$ behaves asymptotically as $\tilde{\mathbb{W}}_n$ in (\ref{Weighted empirical process}), which is a \textit{weighted empirical process}. Therefore, necessary
and sufficient conditions for the convergence of the process $\mathbb{W}_n$ defined in (\ref{Wn process}) are given by the Chibisov-O'Reilly theorem (see \citet[p. 462]{Shorack-Wellner}).
\end{remark}

We are now in position to solve the question proposed in \cite{Jager-Wellner-2004}.

\begin{corollary}
In the conditions of Theorem \ref{Theorem BJ}, the statistic $B_n$ in (\ref{Bn}) satisfies that
\begin{equation*}
B_n\cd \sigma^\prime_{K(F,G)}(\mathbb{W})=\sup_{M^+(K(F,G))} \mathbb{W},\quad \text{as }n\to\infty,
\end{equation*}
where $\mathbb{W}$ is given in (\ref{W}) and the set $M^+(\cdot)$ is defined in (\ref{M+M-}).
\end{corollary}

\begin{remark}
Similar results can be stated for the family of test statistics $S_n(s)$ based on $\phi$-divergences introduced by \cite{Jager-Wellner-2007}. Details are omitted.
\end{remark}

\section{Maximum mean discrepancies}\label{Setion.Maximum.mean.discrepancies}

\subsection{Definition and examples}\label{Subsection.MMD.definition}

Let $X$ and $Y$ be two random variables taking values on a topological space $(\mathcal{X},\tau)$ with Borel probability measures $\P$ and $\Q$, respectively. Throughout this section we will use  the notation
$\E_\P (f)$ to detone the mathematical expectation of $f$ with respect to the probability measure $\P$.
We consider a statistic to measure the dissimilarity between $\P$ and $\Q$ (see \cite{Fortet-Mourier} and \cite{Muller}).

\begin{definition}
Let us consider a class $\X$ of measurable functions $f:\mathcal{X}\longrightarrow\R$. The {\rm maximum mean discrepancy} (MMD in short) between $\P$ and $\Q$ with respect to the class $\X$ is defined by
\begin{equation}\label{MMD}
\text{\rm MMD}[\X,\P, \Q]:= \sup_{f\in\X} \left( \E_\P (f)-\E_\Q (f)       \right).
\end{equation}
\end{definition}

To avoid indeterminate forms in the difference between expectations in (\ref{MMD}), it is usually assumed that $\X$ is a subset of $\mathcal{C}(\mathcal{X},\tau)$, the class of bounded and continuous real functions on $\mathcal{X}$. The probability distribution of the variables is usually completely identified with the MMD with respect to $\mathcal{C}(\mathcal{X},\tau)$. In fact,
if $(\mathcal{X},d)$ is a metric space, then $\P=\Q$ if and only if $\E_\P (f)=\E_\Q (f)$, for all $f\in\mathcal{C}(\mathcal{X},d)$ (see \citet[Lemma 9.3.2]{Dudley}).
However, the class $\mathcal{C}(\mathcal{X},d)$ is in general too large to deal with, so that suitable subsets are usually employed in practice. Another possibility is assuming that the functions $f\in \mathcal{F}$ satisfy that $\sup_{x\in\mathcal{X}}|f(x)|/b(x)<\infty$, for a measurable function $b: \mathcal{X}\longrightarrow [1,\infty)$ such that $\E_\P(b)<\infty$ and $\E_\Q(b)<\infty$. For simplicity, in the following we will not mention these necessary integrability requirements and we will assume that $\sup_{f\in \mathcal{F}}\E_\P(f), \sup_{f\in \mathcal{F}}\E_\Q(f)<\infty$.

We observe that when $\X$ is symmetric, that is, $-f\in\X$ whenever $f\in\X$, we have that $\text{MMD}[\X,\P, \Q]=\sup_{f\in\X} \left| \E_\P (f)-\E_\Q (f)\right|$. In other words, the MMD in (\ref{MMD}) is the \textit{integral probability metric} generated by $\X$ (see \cite{Muller}). In \citet[Section 4.4]{Rachev}, it is also said that the metric has a $\zeta$-structure; see \cite{Zolotarev}. In this section we will also assume that  $\X$ is symmetric.

Some frequently used probability metrics can be expressed as $\text{MMD}[\X,\P, \Q]$, for a suitable choice of the set of functions $\X$. In the following examples $X$ and $Y$ are two random variables with distribution functions $F$ and $G$ and associated probability measures $\P$ and $\Q$, respectively. 

\begin{enumerate}
\item[1.] \textsl{Kolmogorov metric.} This distance is $\Vert F-G\Vert_\infty$, which is the integral probability metric generated by $\X=\{ 1_{(-\infty, x]} : x\in\R \}$. Further, it is also generated by the set of all functions of bounded variation $1$ (see \citet[Theorem 5.2]{Muller}).

\item[2.] \textsl{$L^p$ metrics.} For $1\le p < \infty$, this metric is defined by $d_p(F,G):=\Vert  F-G\Vert_p$  ($\Vert \cdot\Vert_p$ being the usual $L^p$-norm). When $X$ and $Y$ are integrable, $d_p$ admits the dual representation (see \citet[p. 73]{Rachev}) $d_p(F,G)=\text{MMD}[\X_p,\P, \Q]$, where $\mathcal{F}_p$ is the class of all Lebesgue a.e. differentiable functions $f$ such that the derivative $f'$ satisfies $\Vert f' \Vert_q\le 1$ ($q$ being the conjugate of $p$, i.e., $q$ is such that $1/p+1/q=1$).

\item[3.] \textsl{Wasserstein metric.}  This distance is a particular and important case of the $L^p$-metric with $p=1$. Its generator is also the class $\X_{\text{W}}\equiv$ the set of functions $f:\R\longrightarrow\R$ satisfying the Lipschitz condition $|f(x)-f(y)|\le |x-y|$, for all $(x,y)\in\R^2$. By the Kantorovich-Rubinstein theorem, $\Vert F-G\Vert_1=\text{MMD}[\X_{\text{W}},\P, \Q]$. 
    In the context of image processing, this metric is called the \textit{earth mover's distance} (see \cite{Rubner-Tomasi-Guibas}). The importance of the Wasserstein metric, as well as its relevance for optimal transport problems, has been summarized in \citet[Section 6]{Villani}.

\item[4.]  \textsl{Bounded Lipschitz metric.} This metric (see \citet[p. 29]{Huber}) is the integral probability metric generated by $\mathcal{F}_{\text{BL}}:=\{ f : \Vert f \Vert_{\text{BL}} \le 1 \}$, where $\Vert f \Vert_{\text{BL}}:=\Vert f\Vert_{\text{L}}+\Vert f\Vert_\infty$ and $\Vert \cdot\Vert_{\text{L}}$ is the Lipschitz norm given by
$$\Vert f\Vert_{\text{L}}:=\sup_{x\ne y \in \R}\frac{|f(x)-f(y)|}{|x-y|}.$$

\item[5.]  \textsl{Zolotarev ideal metrics of order $r$.} For $r\in \mathbb{N}$, let $\mathcal{Z}_r$ be the class of $(r-1)$-times continuously differentiable functions $f:\R\longrightarrow\R$ satisfying the Lipschitz condition $|f^{(r-1)}(x)-f^{(r-1)}(y)|\le |x-y|$, for all $(x,y)\in\R^2$. (Here we use the notation $f^{(0)}\equiv f$.) The class $\mathcal{Z}_r$ can also be substituted by the set
of functions $f$ having $r$-th derivative $f^{(r)}$ a.e. and such that $|f^{(r)}|\le 1$ a.e. The metric $\zeta_r\equiv\text{MMD}[\mathcal{Z}_r,\P, \Q]$ is called the \textit{Zolotarev metric of order $r$} (see \cite{Rachev} for a general reference and properties of these distances). Convergence in $\zeta_{r}$-metric implies weak convergence plus convergence of the $r$-th absolute moment. Zolotarev metrics have been
used in \cite{Rao} to obtain a CLT for independent, non-identically distributed random variables. As mentioned in \citet[Section 15]{Rachev}, the case $r=2$ is appropriate for investigating some
ageing properties of lifetime distributions. In \cite{Baillo}, $\zeta_2$ has also been used to generate new distance measures for classifying X-ray astronomy data into stellar classes. The metric $\zeta_3$ has been considered in the context of distributional recurrences (see \cite{Neininger-Ruschendorf-1} and \cite{Neininger-Ruschendorf-2}).

\item[6.] \textsl{Zolotarev metric of order $r$ in $L^p$}: For $r\in \mathbb{N}$, and $1\le p \le \infty$, the metric $\zeta_{r,p}$ is generated by $\mathcal{Z}_{r,p}$, the set of functions $f:\R \longrightarrow\R$ for which $f^{(r+1)}$ exists and satisfies $\Vert f^{(r+1)}\Vert_q\le 1$, where $q$ is the conjugate of $p$. Note that $\zeta_{r,1}\equiv\zeta_{r+1}$ (the Zolotarev ideal metric of order $r+1$).
    In risk theory, the metrics $\zeta_{1,\infty}$ and $\zeta_{1,1}$ are respectively called the \textit{stop-loss distance} and the \textit{integrated stop-loss distance} (see \cite{Denuit et al}).

\item[7.] \textsl{Kernel distances}: When $\mathcal{F} = \{ f : \Vert f \Vert_{\mathcal{H}} \le 1\}$ is the unit ball in a reproducing
kernel Hilbert space $\mathcal{H}$, the associated MMD is called \textit{kernel distance}.

\end{enumerate}

\subsection{An asymptotic result for the MMD over Donsker classes}\label{Subsection.MMD.general}

The use of the empirical counterpart of the MMD was already considered in \cite{Fortet-Mourier} and it has been extensively employed in machine learning when $\mathcal{F}$ is the unit ball in a reproducing kernel Hilbert space (RKHS) (see \cite{Gretton-et-al-2012}). In \cite{Sriperumbudur et al}, the authors showed the consistency and rate of convergence of some estimators of various integral probability metrics. The asymptotic behaviour of an estimator of the Zolotarev metric of order $r$ in $L^p$ has been discussed in \cite{Carcamo}. Here we provide a general result regarding the estimation of the MMD. We only consider the two sample case as this situation is the most frequently considered in the literature, but similar results can be obtained in the one sample case.

Let $X_1,\dots,X_n$ and $Y_1,\dots,Y_m$ be two independent random samples from $X$ and $Y$ with probability measures $\P$ and $\Q$, respectively. We denote by $\P_n$ and $\Q_m$ the empirical measures associated with these samples, that is, $\P_n=n^{-1}\sum_{i=1}^n \delta_{X_i}$ and $\Q_m=m^{-1}\sum_{j=1}^m \delta_{Y_j}$, where $\delta_a$ stands for the Dirac delta at the point $a$. Given a class of functions $\mathcal{F}$, the empirical counterpart of $\text{MMD}[\X,\P,\Q]$ in (\ref{MMD}) is given by
\begin{equation}\label{MMD-estimator}
\text{MMD}[\X,\P_n,\Q_m]=\sup_{f\in\X} \left( \frac{1}{n}\sum_{i=1}^n f(X_i) - \frac{1}{m}\sum_{j=1}^m f(Y_j) \right).
\end{equation}

In this section we are interested in the asymptotic behaviour of the quantity
\begin{equation}\label{Mnm}
M_{m,n}:=\sqrt{N} \left( \text{MMD}[\X,\P_n,\Q_m] -  \text{MMD}[\X,\P,\Q] \right),\quad \text{with}\quad N\equiv\frac{nm}{n+m}.
\end{equation}
We observe that $M_{m,n}$ is precisely $D_{n,m}(\sigma)=D_\sigma(D,\mathbb{D}_{n,m},r_{n,m})$ in (\ref{problem}), where the underlying space is $\X=\X$; the target functional is $D\in \ell^\infty(\X)$ given by
\begin{equation}\label{D-RKHS}
D(f):=\E_\P(f)-\E_\Q (f),\quad f\in\X;
\end{equation}
its estimator is
\begin{equation*}
\mathbb{D}_{n,m}(f):= \E_{\P_n}(f)-\E_{\Q_m} (f)=\frac{1}{n}\sum_{i=1}^n f(X_i) - \frac{1}{m}\sum_{j=1}^m f(Y_j), \quad f\in\X;
\end{equation*}
and $r_{n,m}:=\sqrt{N}$. Therefore, from Theorem \ref{Theorem.main}, to derive the asymptotic distribution of $M_{m,n}$ in (\ref{Mnm}) we only need to study the weak convergence in $\ell^\infty (\X)$ of the process $r_{n,m}(\mathbb{D}_{n,m}-D)=:\mathbb{G}_{n,m}$ given by
\begin{equation}\label{Gnm}
\mathbb{G}_{n,m}:= \sqrt{\frac{m}{n+m}}  \mathbb{G}_{n,\P} - \sqrt{\frac{n}{n+m}}  \mathbb{G}_{m,\Q},
\end{equation}
where
\begin{equation*}
\mathbb{G}_{n,\P}:=  \sqrt{n}(\P_n-\P)\quad \text{and}\quad \mathbb{G}_{m,\Q}:=\sqrt{m}(\Q_m-\Q)
\end{equation*}
are two independent $\X$-indexed empirical processes associated with $\P$ and $\Q$, respectively. In other words, for $f\in\X$, we have that
\begin{equation*}
\mathbb{G}_{n,\P}(f)=n^{-1/2} \sum_{i=1}^n (f(X_i)-\E_\P(f))\quad \text{and}\quad \mathbb{G}_{m,\Q}(f)=m^{-1/2} \sum_{j=1}^m (f(Y_j)-\E_\Q(f)).
\end{equation*}

Given a probability measure $\P$, we recall that a class of functions $\mathcal{F}$ is said to be \textit{$\P$-Donsker} if $\mathbb{G}_{n,\P} \rightsquigarrow \mathbb{G}_\P$ in $\ell^\infty(\X)$, where $\mathbb{G}_\P$ is a $\P$-Brownian bridge, that is, $\{ \mathbb{G}_\P (f) : f\in  \X \}$ is a zero-mean Gaussian process with covariance function
\begin{equation*}
\E  \left[\mathbb{G}_\P (f_1)  \mathbb{G}_\P (f_2)\right] =\E_\P (f_1 f_2)-\E_\P(f_1) \E_\P (f_2),\quad f_1,f_2\in \X.
\end{equation*}
Additionally,  $\mathcal{F}$ is \textit{universal Donsker} if it is \textit{$\P$-Donsker}, for every probability measure $\P$ on the sample space.

We observe that whenever $\mathcal{F}$ is $\P$-Donsker, the process $\mathbb{G}_\P $ can be uniquely extended to the $d_\P$-closure of the symmetric convex hull generated by $\mathcal{F}$ (see \citet[Theorem 3.7.28]{Gine-Nickl}), where $d_\P$ is the intrinsic pseudo-metric on $\mathcal{F}$ defined by
\begin{equation*}
d_\P^2(f,g):= \E(\mathbb{G}_\P(f)-\mathbb{G}_\P(g))^2= {\E_\P (f-g)^2-(\E_\P(f-g))^2},\quad f,g\in \mathcal{F}.
\end{equation*}
To simplify the writing, we will not use a different notation for this extension of $\mathbb{G}_\P$.

We are in position to state the main result in this section that determines the asymptotic distribution of the statistic $M_{n,m}$ in (\ref{Mnm}) over Donsker classes.

\begin{theorem} \label{Theorem.MMD}
Let $X$ and $Y$ be two random variables with probability measures $\P$ and $\Q$, respectively. Let us assume that
\begin{enumerate}
\item[(a)] The sampling scheme is balanced, that is, $n/(n + m)\to\lambda$, with $0 < \lambda < 1$, as $n,m\to\infty$.
\item[(b)] The class $\X$ is simultaneously $\P$ and $\Q$-Donsker.
\end{enumerate}
We consider the metric $d$ on $\mathcal{F}$ given by
\begin{equation}\label{metric.d.sum}
d(f,g):= \sqrt{\E_\P(f-g)^2} + \sqrt{\E_\Q (f-g)^2}, \quad f,g\in\mathcal{F}.
\end{equation}
We have that $(\mathcal{F},d)$ is a totally bounded metric space, the function $D$ in (\ref{D-RKHS}) belongs to $\mathcal{C}_u(\mathcal{F},d)$ and the statistic $M_{n,m}$ defined in (\ref{Mnm}) satisfies that
\begin{equation*}
M_{n,m}\cd \sup_{\bar M^+ (D,d)}   \mathbb{G},
\end{equation*}
where $\mathbb{G}:=\sqrt{1-\lambda} \mathbb{G}_\P-\sqrt{\lambda} \mathbb{G}_\Q $ is a zero-mean Gaussian process with $\mathbb{G}_\P$ and $\mathbb{G}_\Q$ two independent $\X$-indexed Brownian bridges associated with $\P$ and $\Q$, respectively, and
\begin{equation*}
\bar M^+ (D,d):=\left\{   f\in (\bar{\mathcal{F}},d) : \E_\P (f)-\E_\Q (f) = \text{\rm MMD}[\X,\P,\Q]      \right\}
\end{equation*}
with $\bar{\X}$ being the $d$-completion of $\X$.
\end{theorem}

\begin{proof}
First, from (a) and (b) we have that $\mathbb{G}_{n,m} \cd  \mathbb{G}$, where $\mathbb{G}_{n,m} $ is in (\ref{Gnm}). Hence, by Theorem \ref{Theorem.main}, $M_{n,m}\cd \sigma_D^\prime (\mathbb{G})$.
Now, as $\X$ is $\P$ and $\Q$-Donsker, the pseudo-metric spaces $(\X,d_\P)$ and $(\X,d_\Q)$ are totally bounded, where $d_\P$ and $d_\Q$ are the natural pseudo-metrics given by
$d_{\text{S}}^2(f,g):= {\E_\text{S} (f-g)^2-(\E_\text{S}(f-g))^2}$, for $\text{S}\in\{ \P,\Q\}$ and $f,g\in \mathcal{F}$ (see \citet[Remark 3.7.27]{Gine-Nickl}). Further,  $\mathbb{G}_\P\in \mathcal{C}_u(\X,d_\P)$ and $\mathbb{G}_\Q\in \mathcal{C}_u(\X,d_\Q)$ a.s. Now, as the class $\X$ is bounded in $L^1(\P)$ and $L^1(\Q)$ (i.e., $\sup_{f\in\mathcal{F}} | \E_\P (f)|, \sup_{f\in\mathcal{F}} | \E_\Q (f)| <\infty$) and $(\X,d_\P)$, $(\X,d_\Q)$ are totally bounded, using the same ideas as in the proof of \citet[Theorem 3.7.40, p. 262]{Gine-Nickl} we conclude that $(\X,d_{L^2(\P)})$ and $(\X,d_{L^2(\Q)})$ are also totally bounded, where  $d_{L^2(\text{S})}^2(f,g):=\E_\text{S}(f-g)^2$ ($f,g\in \mathcal{F}$ and $\text{S}\in\{ \P,\Q\}$). It is easy to check that this implies that $(\X,d)$ is totally bounded, where $d$ is in (\ref{metric.d.sum}). On the other hand, by Cauchy-Schwarz inequality, we have that $|D(f)-D(g)|\le d(f,g)$ and hence $D\in \mathcal{C}_u(\X,d)$.
Further, the paths of $\mathbb{G}$ are in $\mathcal{C}_u(\mathcal{F},d)$ a.s. since $d_\P,d_\Q\le d$. Therefore, the conclusion follows by applying Corollary \ref{Corollary.Totally.Compact} (b).
\end{proof}

Condition (b) in Theorem \ref{Theorem.MMD} is the key assumption that has to be checked to apply the previous result. In other words, we have to ensure that $\X$ is $\P$ and $\Q$-Donsker. There are many results in the literature on empirical proceses guaranteeing that a class of functions is Donsker (see \cite{van der Vaart-Wellner}). For instance, it is well-known that the set of indicators generating the Kolmogorov distance is universal Donsker. The unit ball for the Bounded Lipschitz metric is $\P$-Donsker whenever $\P$ has some finite moments (see \citet[Corollary 5 and Remark 2]{Nickl-Potscher}). In the same work, \cite{Nickl-Potscher} showed that bounded subsets of general function spaces defined over $\R^d$ are Donsker under some appropriate conditions on the underlying probability measure. Examples include (weighted) Besov, Sobolev, H\"{o}lder, and Triebel type spaces. Some of these results have been extended in \cite{Sriperumbudur-2016}.

\section*{Acknowledgements}
This research has been supported by the Spanish MCyT grant MTM2016-78751-P. We thank Holger Dette for pointing out to us the reference \cite{Dette-Kokot-Aue}. The first author thanks Carlos Mora-Corral (Department of Mathematics, Universidad Aut\'onoma de Madrid) for showing to him the counterexample in Section \ref{Subsection.Weakly.compact} that proves that the linear extension $\tilde g$ of a function $g\in \mathcal{C}_{pl}(\X,d_{\mathcal{B}})$ is not necessarily continuous. We are also indebted to the reviewers, AE and Editor for their appropriate and constructive suggestions and the references \cite{Beare-Moon}, \cite{Kaido}, \cite{Seo} and \cite{Beare-Shi}. Their comments have led to a improved version of the original manuscript.

\end{document}